\documentclass[12pt]{amsart} 
\usepackage{amsfonts,graphics,amsmath,amsthm,amsfonts,amscd, amssymb,amsmath,latexsym,enumitem}

\usepackage{epsfig}
\usepackage{flafter}
\usepackage{comment}
\usepackage{color}
\usepackage[abbrev,alphabetic]{amsrefs}
\usepackage[all,cmtip]{xy}
\usepackage{tikz-cd}
\usetikzlibrary{cd}

\usepackage{color}
\usepackage{soul,xcolor} 
\setstcolor{red}
\definecolor{red}{rgb}{0.6,0,0}

\theoremstyle{plain}
\newtheorem{theorem}{Theorem}[section]
\newtheorem{corollary}[theorem]{Corollary}
\newtheorem{lemma}[theorem]{Lemma}
\newtheorem{claim}[theorem]{Claim}

\newtheorem{proposition}[theorem]{Proposition}
\newtheorem{definition-lemma}[theorem]{Definition-Lemma}

\newtheorem{defn}[theorem]{Definition}
\newtheorem{setting}[theorem]{Setting}
\newtheorem{remark}[theorem]{Remark}

\def\ideal#1.{I_{#1}}
\def\ring#1.{\mathcal {O}_{#1}}
\def\fring#1.{\hat{\mathcal {O}}_{#1}}
\def\proj#1.{\mathbb P(#1)}
\def\pr #1.{\mathbb P^{#1}}
\def\af #1.{\mathbb A^{#1}}
\def\Hz #1.{\mathbb F_{#1}}
\def\Hbz #1.{\overline{\mathbb F}_{#1}}
\def\pic#1.{\operatorname {Pic}\,(#1)}
\def\pico#1.{\operatorname{Pic}^0(#1)}
\def\picg#1.{\operatorname {Pic}^G(#1)}
\def\ner#1.{NS (#1)}
\def\rdown#1.{\llcorner#1\lrcorner}
\def\rup#1.{\ulcorner#1\urcorner}
\def\cone#1.{\operatorname {NE}(#1)}

\def\ccone#1.{\overline{\operatorname {NE}}(#1)}
\def\coef#1.{\frac{(#1-1)}{#1}}
\def\vit#1.{D_{\langle #1 \rangle}}
\def\mm#1.{\overline {M}_{0,#1}}
\def\H1#1.{H^1(#1,{\ring #1.})}
\def\ac#1.{\overline {\mathbb F}_{#1}}

\def\adj#1.{\frac {#1-1}{#1}}
\def\spn#1.{\overline{#1}}
\def\ses#1.#2.#3.{0\to #1\to #2\to #3 \to 0}
\def\pek#1.#2.{\Cal P^{#1}(#2)}
\def\plk#1.#2.{\Cal P^{\leq #1}(#2)}
\def\ev#1.{\operatorname{ev_{#1}}}
\def\bminv#1.{(\nu_1,s_1;\nu_2,s_2;\dots ;\nu_{#1},s_{#1};\nu_{r+1})}
\def\zinv#1.{(\nu_1,s_1;\nu_2,s_2;\dots ;\nu_{#1},s_{#1};0)}
\def\iinv#1.{(\nu_1,s_1;\nu_2,s_2;\dots ;\nu_{#1},s_{#1};\infty)}
\def\map#1.#2.{#1 \longrightarrow #2}
\def\rmap#1.#2.{#1 \dasharrow #2}
\def\emb#1.#2.{#1 \hookrightarrow #2}

\def\Supp{\operatorname{Supp}}

\def\dim{\operatorname{dim}}

\def\cod{\operatorname{codim}}
\def\deg{\operatorname{deg}}

\def\det{\operatorname{det}}

\def\N{\mathbb N}

\def\e{\Cal E}
\def\ZZ{\mathbb Z}

\def\e1{E_1}
\def\e2{E_2}

\def\OO{\mathcal O}

\newcommand\Q{{\mathbb{Q}}}
\newcommand\R{{\mathbb{R}}}
\begin{document}
\title{The Minimal Model Program for threefolds in characteristic five.}
\author{Christopher Hacon} 
\address{Department of Mathematics \\  
University of Utah\\  
Salt Lake City, UT 84112, USA}
\email{hacon@math.utah.edu}
\author{Jakub Witaszek} 
\address{Department of Mathematics \\  
University of Michigan\\  
Ann Arbor, MI 48109, USA}
\email{jakubw@umich.edu}

\begin{abstract}
We show the validity of the Minimal Model Program for threefolds in characteristic five.
\end{abstract}

\maketitle



\section{Introduction}
One of the fundamental goals of algebraic geometry is to classify all algebraic varieties which, conjecturally, can be achieved by means of the Minimal Model Program (MMP). A major part of the MMP is now known to hold in characteristic zero (see \cite{bchm06}) and in the last few years substantial progress has been achieved in positive characteristic as well. Indeed, it has been shown that the program
is valid for surfaces over excellent base schemes (see \cite{tanaka12,tanaka16_excellent}) and for three-dimensional varieties of characteristic $p > 5$ (\cite{hx13}, see also \cites{ctx13,birkar13,BW14,GNT06,DW19}).

However, little is known beyond these cases and new phenomena discovered by Cascini and Tanaka (\cite{CT06PLT}) suggest that the low characteristic MMP is much more subtle. Moreover, in view of \cite{CTW15a}, it has become apparent that understanding the geometry of low characteristic threefolds is the most natural step towards tackling the MMP in higher dimensions. 

In \cite{HW18}, following some ideas of \cite{GNT06}, we shed some light on the geometry of threefolds in all characteristics $p \leq 5$. In particular, we show that the relative MMP can be run over $\mathbb{Q}$-factorial singularities and in families. As a consequence, we establish, among other things, inversion of adjunction, normality of plt centres up to a universal homeomorphism as well as the existence of Koll\'ar's components and dlt modifications. 

The goal of this article is to extend the Minimal Model Program for threefolds to characteristic $p=5$ in full generality. We believe that the methods developed in this paper will be useful in tackling the MMP for threefolds in characteristics $2$ and $3$ as well as the MMP in higher dimensions.

Our main result is the following.
\begin{theorem} \label{theorem:flips} Let $(X,\Delta)$ be a $\Q$-factorial dlt three-dimensional pair over a perfect field $k$ of characteristic $p=5$. If $f \colon X\to Z$ is a $(K_X+\Delta)$-flipping contraction, then the flip $f^+ \colon X^+\to Z$ exists.
\end{theorem}
Note that this result is known for $p>5$ by \cite{hx13,birkar13,GNT06}. As a corollary of Theorem \ref{theorem:flips}, we get the following results on the MMP in positive characteristic.
\begin{theorem}[Minimal Model Program with scaling] \label{thm:minimal_models} Let $(X,\Delta)$ be a three-dimensional $\Q$-factorial dlt pair over a perfect field $k$ of characteristic $p>3$ and let $f \colon X \to Z$ be a projective contraction. 
Then we can run an MMP with scaling for $K_X+\Delta$ over $Z$. If $K_X+\Delta$ is relatively pseudo-effective, then the MMP terminates with a log minimal model over $Z$. Otherwise, the MMP terminates with a Mori fibre space.
\end{theorem}
In particular, Theorem \ref{thm:minimal_models} shows that Zariski's conjecture on finite generatedness of the canonical ring of smooth varieties is valid for threefolds in characteristic five.

\begin{theorem}[Base point free theorem]  \label{thm:bpf} Let $(X,\Delta)$ be a three\nobreakdash-\hspace{0pt}dimensional klt pair over a perfect field $k$ of characteristic $p>3$ and let $f \colon X \to Z$ be a projective contraction. Let $D$ be a relatively nef $\Q$-Cartier $\Q$-divisor such that $D-(K_X + \Delta)$ is nef and big over $Z$. Then $D$ is semi-ample over $Z$.
\end{theorem}

\begin{theorem}[Cone theorem] \label{thm:cone}
Let $(X,\Delta)$ be a projective $\Q$-factorial three-dimensional dlt pair over a perfect field $k$ of characteristic $p>3$. Then there exists a countable number of rational curves $\Gamma_i$ such that
\begin{itemize}
	\item $\overline{\mathrm{NE}}(X) = \overline{\mathrm{NE}}(X)_{K_X+\Delta \geq 0} + \sum_i \R[\Gamma_i]$,
	\item $-6 \leq (K_X+\Delta) \cdot \Gamma_i < 0$,
	\item for any ample $\R$-divisor $A$
		\[
			(K_X+\Delta+A) \cdot \Gamma_i \geq 0
		\]
		holds for all but a finitely many $\Gamma_i$, and so
	\item the rays $\R[\Gamma_i]$ do not accumulate inside $\overline{\mathrm{NE}}(X)_{K_X+\Delta<0}$.
\end{itemize}
\end{theorem}
The above results, in this generality, were proven in \cites{birkar13,BW14} (cf.\ \cite{keel99,hx13,ctx13}) contingent upon the existence of flips with standard coefficients. Hence, they follow immediately from Theorem \ref{theorem:flips}. There are many other results around the MMP (cf.\ \cite{birkar13, BW14,waldronlc,GNT06,HNT17}) that generalise to characteristic five in view of Theorem \ref{theorem:flips}.

\subsection{The idea of the proof of Theorem \ref{theorem:flips}}
For simplicity, we suppose in this subsection that the divisorial centres of the dlt pairs we consider are normal. This is not far from the truth, as these divisorial centres are normal up to a universal homeomorphism (see \cite[Theorem 1.2]{HW18}). 

By the same argument as in \cite[Theorem 6.3]{birkar13} we can suppose that the coefficients of $\Delta$ are standard. By perturbation and reduction to pl-flips, we can assume that $\Delta = S +B$ and $(X,S+B)$ is plt, where $S$ is an irreducible divisor. Let $f \colon X \to Z$ be a pl-flipping contraction. The proof of the existence of flips for threefolds in characteristic $p>5$ (\cite{hx13}) consists of two steps:
\begin{enumerate}
	\item showing that the flip of $f$ exists if $(X,S+B)$ is relatively purely F-regular, and
	\item showing that $(X,S+B)$ is relatively purely F-regular when $p>5$.
\end{enumerate}
The first step holds in every characteristic. Unfortunately, the second statement is false  for $p \leq 5$ in general. To circumvent this problem we construct pl-flips by a mix of blow-ups, contractions, and pl-flips admitting dlt $6$-complements.  
\begin{proposition} \label{proposition:building_blocks} (cf. Proposition \ref{proposition:building_blocks_qdlt}) Let $(X,S+B)$ be a $\Q$-factorial plt three-dimensional pair with standard coefficients over a perfect field $k$ of characteristic $p>3$,  where $S$ is an irreducible divisor, and let $f \colon X\to Z$ be a pl-flipping contraction. Assume that there exists a dlt $6$-complement $(X,S+B^c)$ of $(X,S+B)$. Then the flip $f^+ \colon X^+\to Z$ exists.
\end{proposition}
Let $C = \mathrm{Exc}(f)$. For simplicity, assume that $C$ is irreducible. We split the proof of this proposition into three cases:
\begin{enumerate}
\item $(X,S+B^c)$ is plt in a neighborhood of $C$, or
\item $C \cdot E < 0$ for a divisor $E \subseteq \lfloor B^c \rfloor$, or
\item $C \cdot E \geq 0$ for a divisor $E \subseteq \lfloor B^c \rfloor$ intersecting $C$.
\end{enumerate}
In Case (1), write $K_S + B_S = (K_X + S+B)|_S$ and 
\[
K_S + B^c_S = (K_X+S+B^c)|_S.
\]
Since $(X,S+B^c)$ is plt along $C$, we get that $(S,B^c_S)$ is klt along $C$. Our key observation is the following: \emph{if a
 birational log Fano contraction of a surface pair  with standard coefficients in characteristic $p>3$ admits a klt $6$-complement, then it is relatively F-regular} (Proposition \ref{proposition:no_klt_complements}). Therefore, $(X,S+B)$ is relatively purely F-regular by F-adjunction, and so the flip exists by \cite{hx13} (see the aforementioned Step (1)). This is the main part of our arguments which we are unable to generalise to characteristic three. On the other hand, one might expect some analogue of this statement to hold in higher dimensions for all bounded complements and $p \gg 0$. 

In Case (2) we can construct the flip explicitly as the closure of $X$ under the rational map defined by a pencil of sections spanned by $kS$ and $lE$ for some $k,l \in \N$ such that $kS \sim lE$.  

In Case (3), assume for simplicity that $S$ and $E$ are the only log canonical divisors of $(X,S+B^c)$. Then, we can show that $(X,S + B^c-\epsilon E)$ is relatively F-split over $Z$ for $0 < \epsilon < 1$ by F-adjunction applied to $S$ and $S\cap E$. In fact, with a bit more work one can show that it is relatively purely F-regular, and thus $(X,S+B)$ is so as well.  Hence the flip exists by \cite{hx13} as in Case (1).\\

In view of Proposition \ref{proposition:building_blocks}, it is important to construct complements of pl-flipping contractions. By standard arguments, $(S,B_S)$ admits an $m$-complement $(S,B_S^c)$ for $m \in \{1,2,3,4,6\}$, and the following result shows that we can lift it to an $m$-complement of $(X,S+B)$.
\begin{theorem} \label{thm:lift_complements} Let $(X,S+B)$ be a three-dimensional $\Q$-factorial plt pair with standard coefficients defined over a perfect field of characteristic $p>2$ and let $f \colon X \to Z$ be a flipping contraction such that 
$-(K_X+S+B)$ and $-S$ are $f$-ample. 

Then there exists an $m$-complement $(X,S+B^c)$ of $(X,S+B)$ in a neighbourhood of ${\rm Exc}\, f$ for some $m \in \{1,2,3,4,6\}$.
\end{theorem}
Although $(X,S+B)$ need not necessarily be relatively purely F-regular, we can still apply F-splitting techniques as we do not need to lift all the sections, but just some very special ones. Note that this result is new even for $p>5$.\\


In order to construct the flip of $f$ from the flips of Proposition \ref{proposition:building_blocks} we argue as follows. Let $(X,S+B^c)$ be an $m$-complement of $(X,S+B)$ for $m \in \{1,2,3,4,6\}$ which exists by Theorem \ref{thm:lift_complements}. Take a dlt modification $\pi \colon Y \to X$ of $(X,S+B^c)$ with an exceptional divisor $E$. Write $K_Y+S_Y + B_Y^c = \pi^*(K_X+S+B^c)$, $K_Y + S_Y+B_Y = \pi^*(K_X+S+B)$, and run a $(K_Y+S_Y+B_Y)$-MMP over $Z$. Note that it could happen that $B_Y$ is not effective, but we can rectify this situation by taking a linear combination of $B_Y$ and $B_Y^c$ (see the proof of Theorem \ref{theorem:flips} for details). By the negativity lemma, if this MMP terminates, then its output is the flip of $X$. Therefore, it is enough to show that all the steps of this MMP can be performed.

The first step of this MMP definitely exists. Indeed, either it is a divisorial contraction which can be shown to exist by {\cites{keel99,HW18}}, or it is a flipping contraction followed by a flip with a dlt $m$-complement which exists by Proposition \ref{proposition:building_blocks} (the $\mathbb Q$-divisor $B_Y$ may not have standard coefficients, so one needs to be a bit more careful; see the proof for details). However, each step of this $(K_Y+S_Y+B_Y)$-MMP is $(K_Y + S_Y + B^c_Y)$-relatively trivial, and so the dlt-ness of $(Y,S_Y+B^c_Y)$ need not be preserved.

To rectify this problem we employ the notion of qdlt singularities, that is log canonical pairs which are quotient singularities at log canonical centres. In fact, Proposition \ref{proposition:building_blocks} holds for qdlt flipping contractions (Proposition \ref{proposition:building_blocks_qdlt}) and we can show the existence of a qdlt modification $\pi \colon Y \to X$ with irreducible exceptional locus (Corollary \ref{cor:qdlt_modification}). Therefore, the output of any divisorial contraction in the $(K_Y+S_Y + B_Y)$-MMP is automatically the flip of $(X,S+B)$. Moreover, the qdlt-ness of $(Y,S_Y + B_Y^c)$ is preserved by flops (Lemma \ref{lemma:flopping_dlt}) except in one special case in which we can construct the flip of $(X,S+B)$ directly.

\section{Preliminaries}
A scheme $X$ will be called a variety if it is integral, separated, and of finite type over a field $k$. Throughout this paper, $k$ is a perfect field of characteristic $p>0$. We refer the reader to \cite{km98} for basic definitions in birational geometry and to \cite{HW17} for a brief introduction to F-splittings. We remark that in this paper, unless otherwise stated, if $(X,B)$ is a pair, then $B$ is a $\mathbb Q$-divisor. For two $\Q$-divisors $A$ and $B$, we denote by $A \wedge B$ the maximal $\Q$-divisor smaller or equal to both $A$ and $B$. We say that $(X,\Delta^c)$ is an \emph{$m$-complement} of $(X,\Delta)$ if $(X,\Delta^c)$ is log canonical, $m(K_X+\Delta^c) \sim 0$, and $\Delta^c \geq \Delta^*$, where $\Delta^* := \frac{1}{m}\lfloor (m+1)\Delta \rfloor$. If $\Delta$ has standard coefficients, then $\Delta^* = \frac{1}{m}\lceil m\Delta\rceil$, and so the last condition is equivalent to $\Delta^c \geq \Delta$. We say that a morphism $f \colon X \to Y$ is a \emph{projective contraction} if it is a projective morphism of quasi-projective varieties and $f_* \mathcal{O}_X = \mathcal{O}_Y$.

Since the existence of resolutions of singularities is not known in positive characteristic in general, the classes of singularities are defined with respect to all birational maps. For example, a log pair $(X, \Delta)$ is \emph{klt} if and only if the log discrepancies are positive for every birational map $Y \to X$. Similarly, \emph{log canonical centres} are defined as images of divisors of log discrepancy zero under birational maps $Y\to X$. These definitions coincide with the standard ones up to dimension three, as log resolutions of singularities are known to exist in this case.

The starting point for the construction of flips is the following result from \cite{hx13}. We say that a projective birational morphism $f \colon X \to Z$ for a $\Q$-factorial plt pair $(X,S+B)$, with $S$ irreducible, is a \emph{pl-flipping contraction} if $f$ is small, $-(K_X+S+B)$ and $-S$ are relatively ample, and $\rho(X/Z)=1$.
\begin{theorem}\label{thm:hx_flips} Let $(X,S+B)$ be a $\Q$-factorial three-dimensional plt pair defined over a perfect field of characteristic $p>0$ with $S$ irreducible. Let $f \colon X \to Z$ be a pl-flipping contraction. Let $g \colon \tilde S \to S$ be the normalisation of $S$ and write $K_{\tilde S} + B_{\tilde S} = (K_X + S + B)|_{\tilde S}$. If $(\tilde S, B_{\tilde S})$ is relatively globally F-regular over $f(S) \subseteq Z$, then the flip of $f$ exists.
\end{theorem}
Note that the condition on the relative global F-regularity of $(\tilde S, B_{\tilde S})$ is equivalent to the relative pure F-regularity of $(X,S+B)$ by F-adjunction.
\begin{proof}
This follows from \cite{hx13} as explained in \cite[Remark 3.6]{HW18}.
\end{proof}
\begin{remark} \label{rem:hx_flips_complements} By \cite[Theorem 3.1]{hx13} (cf.\ Proposition \ref{prop:hacon_xu_surfaces}) the above assumption on F-regularity is always satisfied when $p>5$  and $B$ has standard coefficients. 
\end{remark}

\subsection{Qdlt pairs}
Qdlt singularities will play an important role in this article.
\begin{defn}[{\cite[Definition 35]{dFKX}}] Let $(X,\Delta)$ be a log canonical variety.  We say that $(X,\Delta)$ is qdlt if for every log canonical centre $x \in X$ of codimension $k>0$, there exist distinct irreducible divisors $D_1, \ldots, D_k \subseteq \Delta^{=1}$ such that $x \in V := D_1 \cap \ldots \cap D_k$.
\end{defn}
\begin{remark} \label{remark:generic_point_of_stratum}
 Note that if $(X,\Delta)$ is log canonical and $x$ is a generic point of a stratum $V := D_1 \cap \ldots \cap D_k$ of $\Delta^{=1}$, then ${\rm codim}\, x = k$. Indeed, let $\tilde D_1 \to D_1$ be the normalisation of $D_1$. Then, by adjunction, $(\tilde D_1, \Delta_{\tilde D_1})$ is log canonical where $K_{\tilde D_1} + \Delta_{\tilde D_1} = (K_X+\Delta)|_{\tilde D_1}$. Moreover, by localising at generic points of $D_1 \cap D_l$ and using surface theory, we see that $D_l|_{\tilde D_1} \subseteq \Delta^{=1}_{\tilde D_1}$ have no mutually common components for $2 \leq l \leq k$. Therefore, $x$ is a generic point of $E_2 \cap \ldots \cap E_k$, where $E_l$ are some irreducible components of $D_l|_{\tilde D_1}$. Now the claim follows by induction.\end{remark}

By \cite[Proposition 34]{dFKX}, in characteristic zero the above definition of qdlt singularities is equivalent to saying that $(X,\Delta)$ is locally a quotient of a dlt pair by a finite abelian group preserving the divisorial centres. 
{In positive characteristic, we know the following.
\begin{lemma} Let $(X,\Delta)$ be a $\Q$-factorial qdlt pair of dimension $n \leq 3$ defined over a perfect field of characteristic $p>0$. Then
\begin{enumerate}
\item $(\tilde D, \Delta_{\tilde D})$ is qdlt, where $g \colon \tilde D \to D$ is the normalisation of a divisor $D \subseteq \Delta^{=1}$ and $K_{\tilde D} + \Delta_{\tilde D} = (K_X+\Delta)|_{\tilde D}$,
\item the strata of $\Delta^{=1}$ are normal up to a universal homeomorphism,
\item the log canonical centres of $(X,\Delta)$ coincide with the generic points of the strata of $\Delta^{=1}$.
\end{enumerate}
\end{lemma}
\begin{proof}
We work in a sufficiently small neighbourhood of a point of $X$.

If $n\leq 2$, then the lemma follows by standard results on surface pairs (cf.\ \cite{KollarSing}).  Indeed, a two-dimensional pair $(X,\Delta)$ is qdlt if either it is plt, or $\Delta = C_1+C_2$, $X$ is an $A_n$-singularity, $(X,\Delta)$ is snc when $n=1$, and, when $n>1$, the strict transforms of $C_1$ and $C_2$ intersect the exceptional locus of the minimal resolution of $X$ transversally at single points on the first and the last curve, respectively. Thus we may assume that $n=3$.

First, note that irreducible divisors in $\Delta^{=1}$ are normal up to a universal homeomorphism. Indeed, if $D \subseteq \Delta^{=1}$ is an irreducible divisor, then $(X,\Delta - \lfloor \Delta \rfloor +D)$ is plt and we can apply \cite[Theorem 1.2]{HW18}.

Let $x \in \tilde D$ be a log canonical centre of $(\tilde D, \Delta_{\tilde D})$. Then $g(x)$ is a log canonical centre of $(X, \Delta)$. Indeed, otherwise there exists an non-zero divisor $H$ passing through $g(x)$ and $\epsilon>0$ such that $(X,\Delta+\epsilon H)$ is lc at $g(x)$. Thus, by adjunction, $(\tilde D, \Delta_{\tilde D}+\epsilon H|_{\tilde D})$ is lc at $x$, which is impossible.

Let $k$ be the codimension of $g(x)$ in $X$. By definition of qdlt pairs, there exist divisors $D_1, \ldots, D_k \subseteq \Delta^{=1}$, with $D_1=D$, such that 
\[
g(x) \in D_1 \cap \ldots \cap D_k.
\]
Then $x \in D_2|_{\tilde D} \cap \ldots \cap D_k|_{\tilde D}$, where $D_i|_{\tilde D} \subseteq \Delta^{=1}_{\tilde D}$ for $i\geq 2$ have no mutually common components (cf.\ Remark \ref{remark:generic_point_of_stratum}). Since $x$ is of codimension $k-1$ in $\tilde D$, this shows that $(\tilde D, \Delta_{\tilde D})$ is qdlt at $x$. Hence (1) holds.

As for (2), pick a stratum $V = D_1 \cap \ldots \cap D_k$ of $\Delta^{=1}$. If $k=1$, then we are done by the first paragraph. Otherwise, 
\[
g^{-1}(V) = D_2|_{\tilde D_1} \cap \ldots \cap D_k|_{\tilde D_1}
\]
is a stratum of $\Delta^{=1}_{\tilde D_1}$ where $g \colon \tilde D_1 \to D_1$ is the normalisation of $D_1$ and $K_{\tilde D_1} + \Delta_{\tilde D_1} = (K_X + \Delta)|_{\tilde D_1}$. Note that each $D_l|_{\tilde D_1}$ is irreducible, as otherwise $(X,\Delta)$ admits a log canonical centre of codimension three which is contained in only two divisors, $D_1$ and $D_l$, of $\Delta^{=1}$. By the surface case, $g^{-1}(V)$ is normal up to a universal homeomorphism, and hence so is $V$ as $g$ is a universal homeomorphism.

Now, we deal with (3). Since the images of log canonical centres of the surface pair $(\tilde D, \Delta_{\tilde D})$ in $X$, for the normalisation $\tilde D$ of a divisor $D \subseteq \Delta^{=1}$, are log canonical centres of $(X,\Delta)$, we see that the generic points of the strata of $\Delta^{=1}$ are log canonical centres. If $x \in X$ is a log canonical centre of $(X,\Delta)$ of codimension $k$, then by definition $x \in V := D_1 \cap \ldots \cap D_k$ for $D_1, \ldots, D_k \subseteq \Delta^{=1}$ and $\cod_X(V)=k$ (cf.\ Remark \ref{remark:generic_point_of_stratum}). Thus, $x$ is a generic point of $V$. 
\end{proof}}

The following lemma generalises the inversion of adjunction from \cite[Corollary 1.5]{HW17}.
\begin{lemma}[{Inversion of adjunction}] \label{lem:qdlt_inv_adjunction} Consider a $\mathbb Q$-factorial three-dimensional log pair $(X,S+E + B)$ defined over a perfect field of characteristic $p>0$, where $S$, $E$ are irreducible divisors and $\lfloor B \rfloor = 0$. Write $K_{\tilde S} + C_{\tilde S} + B_{\tilde S} = (K_X+S+E+B)|_{\tilde S}$, where $\tilde S$ is the normalisation of $S$, the divisor $C_{\tilde S} = (S \cap E)|_{\tilde S}$ is irreducible, and $\lfloor B_{\tilde S} \rfloor = 0$.  Assume that $(\tilde S,C_{\tilde S} + B_{\tilde S})$ is plt. Then $(X,S+E+B)$ is qdlt in a neighborhood of $S$.
\end{lemma}
\begin{proof} Assume by contradiction that $(X,S+E+B)$ admits a log canonical centre $Z$ of codimension at least two, different from $C = S\cap E$, and intersecting $S$. Let $H$ be a general Cartier divisor containing $Z$. Then for any $0 < \delta \ll 1$ we can find $0 < \epsilon \ll 1$ such that $(X,S + (1-\epsilon) E + B + \delta H)$ is not lc at $Z$. On the other hand, $(\tilde S, (1-\epsilon')C_{\tilde S} + B_{\tilde S} + \delta H|_{\tilde S})$ is klt for any $0 < \epsilon' \ll 1$ and $0 < \delta \ll 1$. This contradicts \cite[Corollary 1.5]{HW18}.
\end{proof}

We will use qdlt singularities for log pairs with two divisorial centres. In this case, the qdlt-ness is preserved under flops as long as the divisorial centres intersect each other.
\begin{lemma} \label{lemma:flopping_dlt} Let $(X,S_1+S_2+B)$ be a $\mathbb{Q}$-factorial qdlt three-dimensional pair where $S_1$, $S_2$ are irreducible divisors and $\lfloor B \rfloor = 0$. Let 
\[
f \colon (X, S_1+S_2+B) \dashrightarrow (X', S'_1+S'_2 + B')
\]
be a $(K_X+S_1+S_2+B)$-flop of a curve $\Sigma$ for a relative-Picard-rank-one flopping contraction $g \colon X \to Z$. Suppose that $\Sigma \cdot S_1 <0$. Then $(X',S'_1+S'_2+B')$ is qdlt or $S'_1 \cap S'_2 = \emptyset$ in a neigbhourhood of $\mathrm{Exc}(g')$, where $g' \colon X' \to Z$ is the flopped contraction.
\end{lemma}
\begin{proof}
In proving the proposition we can assume that $X$ and $X'$ are sufficiently small neighbourhoods of $\mathrm{Exc}(g)$ and $\mathrm{Exc}(g')$, respectively. Further, we can assume that the flop is non-trivial, and so a strict transform of a $g$-ample divisor is $g'$-anti-ample.

First, consider the case when $\Sigma \cdot S_2 \geq 0$. Pick a connected component $C \subseteq S_1 \cap S_2$ and let $\tilde S_1 \to S_1$ be the normalisation of $S_1$. Since $(X,S_1+S_2+B)$ is qdlt, $C$ is an irreducible curve. We claim that $C$ is not $g$-exceptional. Indeed, otherwise, in view of $\rho(X/Z)=1$, we have $C\cdot S_2 \geq 0$, which contradicts the following calculation:
\[
C \cdot S_2 = C|_{\tilde S_1} \cdot S_2|_{\tilde S_1} = C|_{\tilde S_1} \cdot (\lambda C|_{\tilde S_1}) < 0
\]
where $\lambda >0$. As a consequence, no component of $S_1 \cap S_2$ is contained in ${\rm Exc}\, g$, and so divisorial places over ${\rm Exc}\, g$ have log discrepancy greater than zero with respect to $(X,S_1+S_2+B)$.  Since flopping preserves discrepancies, we get that the codimension two log canonical centres of $(X',S'_1 + S'_2 + B')$ are images of the generic points of $(S_1 \cap S_2) \, \backslash \, {\rm Exc}\, g$, and so they are generic points of $S'_1 \cap S'_2$. Hence, the pair $(X',S_1'+S_2'+B')$ is qdlt.

Therefore, we can assume that $\Sigma \cdot S_2 < 0$. In particular, 
\[
{\rm Exc}\, g = S_1 \cap S_2
\]
up to replacing $X$ by a neighbourhood of ${\rm Exc}\, g$. Indeed, if we pick an irreducible curve $C \subseteq {\rm Exc}\, g$, then $C \cdot S_i < 0$ for $1 \leq i \leq 2$ as $\rho(X/Z)=1$, and so $C \subseteq S_1\cap S_2$. To prove the inclusion in the opposite direction, assume there exists a non-exceptional irreducible curve $C \subseteq S_1 \cap S_2$ which intersects ${\rm Exc}\, g$ at some exceptional irreducible curve $C'$. As above, $C$ is a connected component of $S_1 \cap S_2$, and so $C' \not \subseteq S_1 \cap S_2$. In particular, $C' \cdot S_i \geq 0$ for some $1 \leq i \leq 2$, which is a contradiction.

We aim to show that $S'_1 \cap S'_2 = \emptyset$. By contradiction assume that $S'_1 \cap S'_2 \neq \emptyset$. By the above paragraph, we have that $S'_1 \cap S'_2 \subseteq {\rm Exc}\, g'$. Since $S_2$ is $g$-anti-ample, $S'_2$ is $g'$-ample and $S'_2|_{\tilde S'_1}$ is an exceptional effective relatively ample divisor, where $\tilde S'_1$ is the normalisation of $S'_1$. This is easily seen to contradict the negativity lemma.
\end{proof} 

\subsection{{Surface lemmas}}
We prove a slightly stronger variant of the construction explained in the proof of \cite[Theorem 3.2]{hx13}.
\begin{lemma} \label{lemma:surface_mmp_lemma} Let $(X,B)$ be a two-dimensional klt pair defined over a perfect field of characteristic $p>0$ and let $f \colon X \to Z$ be a projective birational map such that $-(K_X+B)$ is relatively nef. Then there exist an $f$-exceptional irreducible curve $C$ on a blow-up of $X$ and projective birational maps $g \colon Y \to X$ and $h \colon Y \to W$ over $Z$ such that:
\begin{enumerate}
	\item $g$ extracts $C$ or is the identity if $C \subseteq X$,
	\item $(Y, C+ B_Y)$ is plt,
	\item $(W, C_W + B_W)$ is plt and $-(K_{W}+C_W+B_W)$ is ample over $Z$,
	\item $h^*(K_{W}+C_W+B_W) - (K_{Y}+C+B_Y) \geq 0$,
\end{enumerate} 
where $K_{Y}+bC + B_Y = g^*(K_X+B)$ for $C \not \subseteq \Supp B_Y$, $C_W := h_*C \neq 0$, and $B_W := h_*B_Y$.
\end{lemma}
\noindent The variety $W$ is the canonical model of $-(K_{Y}+C+B_Y)$ over $Z$.
\begin{proof}
Let $\Delta$ be as in \cite[Claim 3.3]{hx13}, that is such that $(X,B + \Delta)$ is lc and admits a unique non-klt place $C$ exceptional over $Z$  and $K_X + B + \Delta \sim_{\Q, Z} 0$. Let $g \colon Y \to X$ be the extraction of the unique non-klt place $C$ of $(X,B+\Delta)$, or the identity if $C$ a divisor on $X$ (see the proof of \cite[Theorem 3.2]{hx13}). By construction, (1) and (2) hold.


Let $G := g^*\Delta - g^*\Delta \wedge C$. Note that $g^*\Delta \wedge C = (1-b)C$. Let $h \colon Y \to W$ be the output of a $G$-MMP over $Z$ (which is equivalent to a $-(K_{Y} + C + B_Y)$-MMP). Let $G_W := h_*G$. Now, (4) follows by the negativity lemma. 

To prove (3), notice that since $C$ is not contained in the support of $G$, then $G\cdot C\geq 0$ and so $C$ is not contracted by $Y\to W$. Since 
\[
K_{Y}+C+B_Y+G=g^*(K_X+B+\Delta)\sim _{\Q,Z}0
\]
is plt, it follows that $(W,C_W+B_W+G_W)$ is plt, and hence so is $(W,C_W+B_W)$. Since $W$ is a $G$-minimal model over $Z$, then $-(K_{W}+C_W+B_W)\sim _{\Q,Z}G_W$ is nef, and in particular semiample over $Z$. To conclude the proof of (3), we need to show that $(K_{W}+C_W+B_W) \cdot C_W < 0$. Indeed, if this is true then the associated semiample fibration does not contract $C_W$ and so we can replace $W$ by the image of the associated semiample fibration to make $-(K_{W}+C_W+B_W)$ ample without giving up the plt-ness of $(W,C_W+B_W)$.

Assume by contradiction that $(K_{W}+C_W+B_W) \cdot C_W = 0$. Let $\Gamma$ be an effective $\Q$-divisor constructed as a connected component of 
\[
h^*(K_{W}+C_W+B_W) - (K_{Y}+bC+B_Y)
\]
containing $C$. Since $\Gamma$ is exceptional over $Z$, we have $\Gamma^2 < 0$. This contradicts the following calculation:
\begin{align*}
\Gamma^2 &= \Gamma \cdot (h^*(K_{W}+C_W+B_W) - (K_{Y}+bC+B_Y)) \\
		 &\geq \Gamma \cdot h^*(K_{W}+C_W+B_W) \\
		 &= h_*\Gamma \cdot (K_{W} + C_W + B_W) = 0,
\end{align*}
as $\Supp h_*\Gamma = C_W$.
\end{proof}

The above result allows for a shorter proof of \cite[Theorem 3.1]{hx13}.
\begin{proposition}[{\cite[Theorem 3.1]{hx13}}] \label{prop:hacon_xu_surfaces} With notation as in the above lemma, suppose that $B$ has standard coefficients and $p>5$. Then $(X,B)$ is globally F-regular over $Z$.
\end{proposition}
\begin{proof}
By \cite[Proposition 2.11 and Lemma 2.12]{hx13}, it is enough to show that $(W,C_W + B_W)$ is purely globally F-regular over $Z$,    and so by F-adjunction (see \cite[Lemma 2.10]{HW17}) it is enough to show that $(C, B_{C})$ is globally F-regular, where $K_{C} + B_{C} = (K_W + C_W + B_W)|_{C}$ and $C$ is identified with $C_W$. Since $-(K_{C} + B_{C})$ is ample and $B_{C}$ has standard coefficients, this follows from \cite[Theorem 4.2]{watanabe91}.
\end{proof}
\begin{remark} \label{rem:surfaces_char_5} If $p=5$, then the above proposition holds true unless $B_{C} = \frac{1}{2}P_1 + \frac{2}{3}P_2 + \frac{4}{5}P_3$ for three distinct points $P_1$, $P_2$, and $P_3$ (see \cite[4.2]{watanabe91}).
\end{remark}

In what follows, we will need the following result.
\begin{lemma} \label{lem:surface_char_5_with_special_complement} With notation as in Lemma \ref{lemma:surface_mmp_lemma}, suppose that $p>3$ and $(X,B)$ admits a $6$-complement $(X, E+B^c)$, where $E$ is a non-exceptional irreducible  curve intersecting the exceptional locus over $Z$. Then $(X,B)$ is globally F-regular over $Z$.

\end{lemma}
\noindent Note that we do not assume that $B$ has standard coefficients.
\begin{proof}
As in the proof of Proposition \ref{prop:hacon_xu_surfaces}, it is enough to show that $(W,C_W + B_W)$ is purely globally F-regular over $Z$, and so by F-adjunction (see \cite[Lemma 2.10]{HW17}) it is enough to show that $(C, B_{C})$ is globally F-regular, where $K_{C} + B_{C} = (K_W + C_W + B_W)|_{C}$ and $C$ is identified with $C_W$.

By pulling back the complement to $Y$ and pushing down on $W$, we obtain a sub-lc pair $(W, aC_W+E_W+B^c_W)$ for a (possibly negative) number $a \in \Q$  such that $6(K_W+aC_W+E_W+B^c_W)\sim _Z 0$, a non-exceptional irreducible curve $E_W$ intersecting the exceptional locus over $Z$, and an effective $\Q$-divisor $B^c_W$ such that $E_W + B^c_W \geq B_W$. Let $T_W$ be an effective exceptional anti-ample $\Q$-divisor on $W$ and let $\lambda \geq 0$ be such that the coefficient of $C_W$ in $aC_W+\lambda T_W$ is one. By the Koll\'ar-Shokurov connectedness theorem (see e.g.\ \cite[Theorem 5.2]{tanaka16_excellent}), the pair $(W, aC_W + \lambda T_W + E_W+B^c_W)$ is not plt along $C_W$. In particular, $B^c_{C}$ contains a point with coefficient at least one, where 
\[
(K_W + aC_W + \lambda T_W + E_W + B^c_W)|_{C} = K_{C} + B^c_{C}.
\] 
Since $T_W$ is anti-ample over $Z$, we have that $K_{C} + B^c_{C}$ is anti-nef. In particular, there exists a $\Q$-divisor $B_{C} \leq B
'_C \leq B^c_C$ such that $(C, B'_C)$ is plt (but not klt) and $-(K_C+B'_C)$ is nef. 

If $-(K_C+B'_C)$ is ample, then $(C, B'_C)$ is purely F-regular by \cite[Lemma 2.9]{CTW} (applied to perturbations of $(C,B'_C)$), and so $(C,B_C)$ is globally F-regular. 
If $-(K_C+B'_C)$ is trivial, then $a=1$, $\lambda=0$, $6(K_C+B^c_C) \sim 0$, and $(C,B^c_C)$ is plt (but not klt). Since $\mathrm{GCD}(p,6)=1$, \cite[Lemma 2.9]{CTW} implies that $(C,B^c_C)$ is globally F-split, and so $(C,B_C)$ is globally F-regular by \cite[Corollary 3.10]{SS}.
\end{proof}
\subsection{Dual complexes}
Let $(X,\Delta)$ be a three-dimensional dlt pair. Its dual complex $D(\Delta^{=1})$ is a simplex with nodes corresponding to irreducible divisors of $\Delta^{=1}$ and $k$-simplices between $k+1$ nodes corresponding to $k+1$ divisors containing a common codimension $k+1$ locus.

Let $\pi \colon Y \to X$ be a projective birational morphism such that $(Y,\Delta_Y)$ is dlt, where $K_Y + \Delta_Y = \pi^*(K_X+\Delta)$. In characteristic zero one can show, using the weak factorisation theorem, that $D(\Delta_Y^{=1})$ is homotopy equivalent to $D(\Delta^{=1})$. In characteristic $p>0$, the weak factorisation theorem is not known to hold, but a similar result may be obtained by running an MMP and using the proof of \cite[Theorem 19]{dFKX} (cf.\ \cite[Subsection 2.3]{Nakamura19}). 

For the convenience of the reader, we give a direct proof of a consequence of the above result, one, that we will need later. Here, we say that an irreducible divisor $D$ in $\Delta^{=1}$ is an \emph{articulation point}, if $\Delta^{=1} - D$ is disconnected. 
\begin{lemma} \label{lemma:articulation_points} Let $(X,\Delta)$ be a $\Q$-factorial dlt threefold over a perfect field and let $\pi \colon Y \to X$ be a projective birational morphism such that $(Y,\pi^{-1}_*\Delta + E)$ is dlt, where $E$ is the exceptional locus of $\pi$. Write $K_Y + \Delta_Y = \pi^*(K_X+\Delta)$. Let $S$ be an irreducible divisor in $\Delta^{=1}$, and let $S_Y$ be its strict transform. If $S_Y$ is an articulation point, then so is $S$.
\end{lemma} 
\begin{proof}
Assume that $S_Y$ is an articulation point of $D(\Delta_Y^{=1})$ and let $h \colon Y \dashrightarrow X'$ be the output of a $(K_Y + \pi^{-1}_*\Delta + E)$-MMP over $X$ (which we can run by \cite[Theorem 1.1]{HW17}). Further, let 
\[
\Delta_{X'} := h_*\Delta_Y = h_*( \pi^{-1}_* \Delta + E)
\]
and $S_{X'} := h_*S_Y$. First, we show that $S_{X'}$ is an articulation point of $D(\Delta_{X'}^{=1})$. To this end, we claim that there is a natural inclusion of dual complexes 
\[
D(\Delta_{X'}^{=1}) \subseteq D(\Delta_Y^{=1})
\]
which identifies the nodes of these dual complexes. Indeed, decompose $h \colon Y \dashrightarrow X'$ into flips and divisorial contractions of the $(K_Y + \pi^{-1}_*\Delta + E)$-MMP: 
\[
	Y =: Y_1 \overset{h_1}{\dashrightarrow} Y_2 \overset{h_2}{\dashrightarrow} \ldots \overset{h_{k-1}} \dashrightarrow Y_k:= X'.
\]
Denote the strict transforms of $\Delta_Y$ by $\Delta_{Y_1}, \ldots, \Delta_{Y_k}$ and $\Delta_{X'}$, respectively, and the projections to $X$ by $\pi_i \colon Y_i \to X$. Note that $K_Y + \pi^{-1}_*\Delta + E \sim_{X, \Q} a_1E_1 + \ldots +a_m E_m$ for all exceptional divisors $E_1, \ldots, E_m \not \subseteq \Delta^{=1}_Y$ and $a_1, \ldots, a_m>0$, and so, by the negativity lemma, this MMP contracts exactly those divisors in $E$ which are not contained in $\Delta_Y^{=1}$. In particular, it preserves the nodes of $D(\Delta_Y^{=1})$. 

Set $\overline{\Delta}_{Y_l} = (\pi^{-1}_l)_*\Delta + \mathrm{Exc}(\pi_l)$. Note that there is no log canonical centre of $(Y_l, \Delta_{Y_l})$ contained in $\mathrm{Exc}((h_{l-1})^{-1})$ by the negativity lemma. Indeed,  suppose that there is such a centre $Z$. Then $Z$ is also a log canonical centre of $(Y_l, \overline{\Delta}_{Y_l})$, and there exists an exceptional divisorial place $E_Z$ over $Y_l$ with centre $Z$ such that $a_{E_Z}(Y_l,\overline{\Delta}_{Y_l})=0$. Since $h_{l-1}$ is not an isomorphism over the generic point of $Z$, \cite[Lemma 3.38]{km98} implies that 
\[
0 \leq a_{E_Z}(Y_{l-1}, \overline{\Delta}_{Y_{l-1}}) < a_{E_Z}(Y_l,\overline{\Delta}_{Y_l})=0,
\]
which is a contradiction. 

Now, projecting by $h_{l-1}$ provides a bijection
\[
\{Z_{l-1} \in \mathrm{LCC}(Y_{l-1}, \Delta_{Y_{l-1}}) \mid Z_{l-1} \not \subseteq \mathrm{Exc}(h_{l-1}) \} \leftrightarrow  \{Z_l \in \mathrm{LCC}(Y_l, \Delta_{Y_l}) \} 
\] 
for any $1 < l \leq k$. In particular, this induces an inclusion $D(\Delta_{Y_l}^{=1}) \subseteq D(\Delta_{Y_{l-1}}^{=1})$, and so the claim holds and $S_{X'}$ is an articulation point.

Let $\pi' \colon X' \to X$ be the induced morphism. Note that $K_{X'} + \Delta_{X'} = (\pi')^*(K_X+\Delta)$ and the divisor $\mathrm{Exc}(\pi')$ is contained in $\Delta_{X'}^{=1}$. First, we show that
\[
\pi'(\Delta^{=1}_{X'}-S_{X'}) \subseteq \Supp(\Delta^{=1}-S).
\] 
To this end pick an irreducible divisor $D \subseteq \Supp(\Delta_{X'}^{=1}-S_{X'})$. Then $\pi'(D)$ is a log canonical centre of $(X,\Delta)$, and so, since $(X,\Delta)$ is dlt, there exists a divisor $S' \subseteq \Supp(\Delta^{=1}-S)$ such that $\pi'(D) \subseteq S'$. This shows the above inclusion.

Now, note that
\[
\pi'|_{\Supp(\Delta^{=1}_{X'}-S_{X'})} \colon \Supp(\Delta^{=1}_{X'}-S_{X'}) \to \Supp(\Delta^{=1}-S)
\] 
has connected fibres. { Indeed, $\mathrm{Exc}(\pi') \subseteq \Supp(\Delta^{=1}_{X'}-S_{X'})$ and $\pi'$ has connected fibres.} Therefore, $\Delta^{=1}_{X'}-S_{X'}$ is disconnected if and only if so is $\Delta^{=1}-S$. In particular, $S$ is an articulation point.


\end{proof}

\section{Complements on surfaces}
The following proposition is fundamental in showing that flips admitting a qdlt $6$-complement exist. Note that every two-dimensional log pair with standard coefficients and which is log Fano with respect to a projective birational map admits a relative $m$-complement for $m\in \{1,2,3,4,6\}$ (cf.\ \cite[Theorem 3.2]{hx13}).
\begin{proposition} \label{proposition:no_klt_complements} Let $(S,B)$ be a two-dimensional klt pair with standard coefficients defined over a perfect field of characteristic $p>3$ and let $S \to T$ be a birational contraction such that $-(K_S+B)$ is relatively nef but not numerically trivial.  Assume that $(S,B)$ is not relatively globally F-regular over $T$. 

Then every $6$-complement of $(S,B)$ is non-klt and has a unique non-klt valuation which is exceptional over $T$.
\end{proposition}
\begin{proof}
We work over a sufficiently small neighbourhood of a point $t \in T$. By Lemma \ref{lemma:surface_mmp_lemma}, there exist an irreducible, exceptional over $T$, curve $C$ on a blow-up of $S$ and projective birational maps $g \colon Y \to S$ and $h \colon Y \to W$ over $T$ such that
\begin{enumerate}
	\item $g$ extracts $C$ or is the identity if $C \subseteq S$,
	\item $(Y, C+ B_Y)$ is plt,
	\item $(W, C_W + B_W)$ is plt and $-(K_{W}+C_W+B_W)$ is ample over $T$,
\end{enumerate} 
where $C_W := h_*C \neq 0$, $B_W := h_*B_Y$, and $K_{Y}+bC + B_Y = g^*(K_S+B)$ for $C \not \subseteq \Supp B_Y$.

By Remark \ref{rem:surfaces_char_5}, $(K_{W} + C_W + B_W)|_{C_W} = K_{C_W} + \frac{1}{2}P_1 + \frac{2}{3}P_2 + \frac{4}{5}P_3$ for some three distinct points $P_1$, $P_2$, and $P_3$.  \\

Now, let $(S,B^c)$ be any $6$-complement of $(S,B)$. By the negativity lemma $\Supp(B^c - B)$ contains a non-exceptional curve. Let $K_{Y} + aC + B_Y^c = g^*(K_S + B^c)$, where $C \not \subseteq \mathrm{Supp}\, B_Y^c$, and let $B_W^c := h_*B_Y^c$.  Since $6(K_S+B^c)\sim_{T} 0$ is lc, we get that 
\[
(W,aC_W + B_W^c)
\]
is sub-lc and $6(K_{W}+a{C_W}+B_W^c) \sim_{T} 0$. In particular $6B_W^c$ is an integral divisor. Moreover, $B_W^c \geq B_W$ as $B^c \geq B$.\\


To prove the proposition it is now enough to show that $a=1$. Indeed, in this case $-(K_{W}+C_W+B_W^c) \sim_{\Q, T}0$ and by the Koll\'ar-Shokurov connectedness lemma, the non-klt locus of $(W,C_W+B_W^c)$ is connected. The only $6$-complement of 
\[
(C_W,\frac{1}{2}P_1 + \frac{2}{3}P_2 + \frac{4}{5}P_3)
\]
is $(C_W,\frac{1}{2}P_1 + \frac{2}{3}P_2 + \frac{5}{6}P_3)$, so $(W,C_W+B_W^c)$ is plt along $C_W$ by adjunction, and connectedness of the non-klt locus implies that $(W,C_W+B_W^c)$ is in fact plt everywhere. In particular, $(S,B^c)$ admits a unique exceptional over $T$ non-klt valuation.

In order to prove the proposition, we assume that $a<1$ and derive a contradiction. We will not need to refer to $(S,B)$ or $(Y,aC+B_Y)$ any more, so, for ease of notation, we replace $C_W$, $B_{W}$, and $B^c_W$  by $C$, $B$, and $B^c$, respectively.

If  $(B^c-B) \cdot C \ne  0$, then Lemma \ref{lemma:auxiliary_complements} applied to $(W,C+B^c)$ implies that $(K_W+C+B^c) \cdot C = 0$. This is impossible, because
\[ 
(K_W+C+B^c) \cdot C < (K_W + aC + B^c) \cdot C = 0.
\]
Hence, we can assume that $(B^c - B) \cdot C = 0$. Since $\Supp(B^c-B)$ contains a non-exceptional curve, the exceptional locus over $T$ cannot be irreducible, and so there exists an irreducible exceptional curve $E \neq C$ such that $E \cap C \neq \emptyset$. Since  $E\cong \mathbb P ^1$ and $E^2<0$, we may contract $E$ over $T$. Let $f \colon W \to W_1$ be a contraction of $E$, and let $C_1$, $B^c_1$ be the strict transforms of $C$ and $B^c$. We have that
\[
(K_W + C + B^c) \cdot E > (K_W + aC + B^c) \cdot E = 0,
\] 
and hence for some $t>0$ and with the natural identification $C \simeq C_1$:
\begin{align*}
(K_{W_1} + C_1 + B^c_1)|_{C_1} &= f^*(K_{W_1} + C_1 + B^c_1)|_C \\
&= (K_W + C + B^c + tE)|_C \\
&\geq K_C + \frac{1}{2}P_1 + \frac{2}{3}P_2 + \frac{4}{5}P_3 + tE|_C.   
\end{align*}
As before, $(K_{W_1} + C_1 + B^c_1) \cdot C_1 < (K_{W_1}+aC_1 + B^c_1) \cdot C_1 = 0$. By applying Lemma \ref{lemma:auxiliary_complements} to $(W_1,C_1+B^c_1)$, we get a contradiction again.
\end{proof}

In the following result, it is key that $\Delta$ is non-zero.
\begin{lemma} \label{lemma:auxiliary_complements} Let $(S,C+B)$ be a two dimensional log pair and let $f \colon S \to Z$ be a projective birational morphism such that the irreducible normal divisor $C$ is exceptional and $(K_S+C+B) \cdot C \leq 0$. Assume that $6B$ is an integral divisor and
\[
B_C = \frac{1}{2}P_1 + \frac{2}{3}P_2 + \frac{4}{5}P_3 + \Delta
\]
for distinct points $P_1,P_2,P_3 \in C$ and a non-zero effective $\Q$-divisor $\Delta$, where $(K_S+C+B)|_C = K_C + B_C$. Then $(K_S+C+B) \cdot C = 0$.
\end{lemma}
\begin{proof}
By contradiction assume that $(K_S+C+B) \cdot C < 0$. Since $\frac{1}{2}+\frac{2}{3}+\frac{4}{5} = 2 - \frac{1}{30}$, we get
\[
-\frac{1}{30} < (K_S+C+B) \cdot C < 0.
\]
Set $y_i := \frac{1}{2}, \frac{2}{3}, \frac{4}{5}$ for $i=1,2,3$, respectively, and write 
\[
B_C = x_1P_1 + x_2P_2 + x_3P_3 + \Delta',
\]
where $\Delta' \geq 0$ and $P_i \not \subseteq \Supp \Delta'$. We have that
\[
y_i \leq x_i < y_i + 1/30,
\]
and so $x_i < \frac{4}{5} + \frac{1}{30} = \frac{5}{6}$. Further, $\deg \Delta' < \frac{1}{30}$. By adjunction, $(S,C+B)$ is plt along $C$.

Let $\Gamma_i$ be the intersection matrix of the singularity of $S$ at $P_i$. Recall that $\det \Gamma_i$ is the $\mathbb Q$-factorial index of $P_i$, i.e. for any Weil divisor $D$, it holds that $(\det \Gamma_i)D$ is Cartier at $P_i$  (see \cite[Lemma 2.2]{CTW}). By \cite[Corollary 3.45]{KollarSing}
\[
x_i = 1 - \frac{1}{\det \Gamma_i} + \frac{k}{6\det \Gamma_i}
\] 
for some   integer $k\geq 0$; in particular it is of the form $\frac{m}{6 \det \Gamma_i}$. Moreover, $\det \Gamma_i \leq 5$, as otherwise $x_i \geq \frac{5}{6}$. 

We claim that $x_i=y_i$. If $i \in \{1,2\}$, then $6(\det \Gamma_i)y_i \in \N$ and so either $x_i=y_i$ or 
\[
x_i \geq y_i + \frac{1}{6\det \Gamma_i} \geq y_i + \frac{1}{30}
\]
which is a contradiction.
If $i = 3$, then since $\det \Gamma_3\leq 5$ it is easy to see that
\[
\frac{5\det \Gamma_3-1}{6\det \Gamma_3} \leq \frac{4}{5} \leq x_3 < \frac{5\det \Gamma_3}{6\det \Gamma_3} = \frac{5}{6}.
\]
Since $x_i=\frac m{6\det \Gamma_3}$, it follows that  $x_3=y_3=\frac{4}{5}$ and $\det \Gamma_3=5$.

Hence, $x_i=y_i$ for all $i \in \{1,2,3\}$ and $\Delta' = \Delta$.  In particular, either $\Supp \Delta$ is contained in the smooth locus of $S$ and $\deg \Delta \geq \frac{1}{6} \geq \frac{1}{30}$, or $\deg \Delta$ is bounded from below by the smallest standard coefficient: $1/2$.  In either case, this is a contradiction.
\end{proof}

\section{Lifting complements}
The new building blocks for the low characteristic MMP are flips admitting a qdlt $6$-complement. Therefore, it is fundamental to construct $6$-complements of flipping contractions. This is done by lifting them from divisorial centres as described by Theorem \ref{thm:lift_complements}. Before we move on to the proof of this result, we need to show some results about Frobenius stable sections for $\Q$-divisors.

\subsection{Frobenius stable sections and integral adjunction}
In this subsection we assume the existence of log resolutions of singularities admitting relatively anti-ample effective divisors. In particular, the results of this section are valid up to dimension three. Further, we denote the Frobenius stable sections of a line bundle $L$ with respect to the Frobenius trace map associated to $(X,\Delta)$ by $S^0(X, \Delta; L)$. Note that this space is often denoted by $S^0(X, \sigma(X,\Delta) \otimes L)$. We refer to \cite{schwede14} and \cite{hx13} for the definition and a comprehensive treatment of $S^0$.

Let $(X,\Delta)$ be a positive characteristic log Fano pair. Fix $m \in \N$ and set $A := -(K_X+\Delta)$. We want to study the sections in $H^0(X, \lfloor mA \rfloor)$, which are Frobenius stable with respect to a carefully chosen boundary.

 If $\Delta$ has standard coefficients, then the theory of complements gives a natural candidate: $\Phi := \{(m+1)\Delta\}$. Indeed, in this case, $\lfloor mA \rfloor -(K_X+\Phi) = -(m+1)(K_X+\Delta)$ is ample (see (\ref{eq:s0_integral})), which suggests that one should look at the subspace
\[
S^0(X,\Phi; \lfloor mA \rfloor) \subseteq H^0(X, \lfloor mA \rfloor). 
\]

Since standard coefficients are not stable under log pull-backs or perturbations, we need to work in a more general setting. 
\begin{setting} \label{setting:s0_integral_adjunction}Fix a natural number $m \in \N$ and a perfect field $k$ of characteristic $p>0$. Let $(X,S+B)$ be a sub-log pair which is projective over an affine $k$-variety $Z$ and such that $S$ is a (possibly empty) reduced Weil divisor, $\lfloor B \rfloor \leq 0$, and $A := -(K_X+S+B)$ is nef and big. 

We are ready to define: 
\begin{align*}
	\Phi &:= S + \{(m+1)B\},\\
	D &:= \lceil mB \rceil - \lfloor (m+1)B \rfloor, \text{ and } \\ 
	L &:= \lfloor mA \rfloor + D.
\end{align*}
 For the sake of future perturbations, we choose an effective $\Q$-divisor $\Lambda$ with sufficiently small coefficients, no common components with $S$, and such that $K_X+S+B+\Lambda^m$ is of Cartier index non-divisible by $p>0$, where $\Lambda^m := \frac{1}{m+1}\Lambda$. Such $\Lambda$ exists by Remark \ref{rem:integral_s0_perturbations2}.
\end{setting}
 We call $D$ the \emph{defect} divisor and say that $(X,S+B)$ has \emph{zero defect} if $D=0$. Note that $(X,S+B)$ has zero defect when $B$ has standard coefficients. In general, since $\lfloor B \rfloor \leq 0$, we have 
 \begin{align*} D&=\lceil mB\rceil -\lfloor (m+1)B\rfloor \\ & = \lceil mB - (m+1)B+\{  (m+1)B\}\rceil  \\ &=\lceil -B +\{  (m+1)B\}\rceil \geq 0.\end{align*} Moreover,
\begin{align*} 
\lfloor mA \rfloor &= -m(K_X+S) - \lceil mB \rceil  \\
				   	&= -m(K_X+S) - \lfloor (m+1)B \rfloor - D \\
				   &= K_X + \Phi - (m+1)(K_X+S+B) - D,
\end{align*}
and so
\begin{equation} \label{eq:s0_integral}
L - (K_X+\Phi) = -(m+1)(K_X+S+B) = (m+1)A
\end{equation} 
is nef and big. In particular, $L - (K_X+\Phi+\Lambda) = -(m+1)(K_X+S+B+\Lambda^m)$, and so the Weil index of $K_X+\Phi+\Lambda$ is not divisible by $p$. 

\begin{defn} \label{definition:c0} With notation as above, define $C^0_{\Lambda}(X,S+B; L) := S^0(X, \Phi + \Lambda; L) \subseteq H^0(X, L)$. \end{defn}
\noindent By Noetherianity and the fact that $\Lambda$ is assumed to have sufficiently small coefficients, we can replace $\Lambda$ by any $\Lambda'$ satisfying the assumptions of  Setting \ref{setting:s0_integral_adjunction}, having sufficiently small coefficients, and such that $\Supp \Lambda' = \Supp \Lambda$.

\begin{remark} \label{rem:integral_s0_perturbations2} 
There always exists $\Lambda$ as in Setting \ref{setting:s0_integral_adjunction}. Indeed, we can assume that $K_X$ is such that $S \not \subseteq \Supp (A)$.  Pick a sufficiently ample Cartier divisor $M$, use Serre vanishing to find $M' \sim M$ vanishing along $\Supp (A)$ with high multiplicity but without vanishing along $S$, and set $\Lambda := (m+1)(M' +A)$. Moreover, given such $\Lambda$ we can replace it by $\epsilon \Lambda$ for some $0< \epsilon \ll 1$ by the same argument as in \cite[Lemma 2.10]{witaszek15}.
\end{remark}

The following lemma allows for calculating $C^0$ on a log resolution. 
\begin{lemma} \label{lem:integral_s0_birational} With notation as in Setting \ref{setting:s0_integral_adjunction} suppose that $(X,S+B)$ is plt and has zero defect. Let $\pi \colon Y \to X$ be a projective birational map and set $K_Y+S_Y+B_Y = \pi^*(K_X+S+B)$ with $S_Y:= \pi^{-1}_*S$. Then,
\[
C^0_{\Lambda_Y}(Y,S_Y+B_Y; L_Y) = C^0_{\Lambda}(X, S+B; L),
\] 
where $L_Y$ is defined for $(Y,S_Y+B_Y)$ as in Setting \ref{setting:s0_integral_adjunction}, and $\Lambda_Y := \pi^*\Lambda$.
\end{lemma}
\noindent Note that $L_Y$ is rarely the pullback of $L$. This lemma holds for any $\Q$-divisor $\Lambda_Y$ satisfying the assumptions of Setting \ref{setting:s0_integral_adjunction} and such that $\Supp \Lambda_Y = \Supp \pi^*\Lambda$. 
\begin{proof}
Set $\Lambda^m_Y := \frac{1}{m+1}\Lambda_Y$. Since $(X,S+B)$ is plt, we have that $\lfloor B_Y \rfloor \leq 0$. The subspace $S^0(Y,\Phi_Y + \Lambda_Y; L_Y)$ is given as the image of 
\[
H^0(Y, F^e_* \OO_Y((1-p^e)(K_Y + \Phi_Y+\Lambda_Y) + p^eL_Y)) \to H^0(Y,\OO_Y(L_Y))
\]
for a sufficiently divisible $e>0$. Therefore, it is enough to show the following two identities: $\pi_* \OO_Y(L_Y) = \OO_X(L)$ and 
\begin{equation*}
\pi_* \OO_Y((1-p^e)(K_Y + \Phi_Y + \Lambda_Y) + p^eL_Y) = \OO_X((1-p^e)(K_X + \Phi+\Lambda) + p^eL).
\end{equation*}

We begin by checking the first one. Since $\pi _*L_Y=L$, there is an inclusion $\pi_* \OO_Y(L_Y) \subset \OO_X(L)$. Since $mA_Y + D_Y =\pi^*(mA) + D_Y$ where $D_Y$ is an effective Weil divisor, we have
\[L_Y=\lfloor mA_Y\rfloor +D_Y=\lfloor  \pi^*(mA) \rfloor + D_Y\geq    \pi^*(\lfloor mA \rfloor ) + D_Y =\pi ^* L+D_Y.\]
Here we used the fact that the defect $D=0$. Since $D_Y$ is effective and exceptional,
$\pi _* \OO _Y(D_Y)=\OO _X$. The inclusion $\pi _* \OO _Y(L_Y)\supset \OO _X(L)$ now follows from the projection formula.

We will now show the second one. To this end, we can use (\ref{eq:s0_integral}) to write
\[
(1-p^e)(K_Y + \Phi_Y+\Lambda_Y) + p^eL_Y= (1-p^e)(m+1)(K_Y+S_Y+B_Y+\Lambda^m_Y) + L_Y.
\]

Since $K_Y+S_Y+B_Y+\Lambda^m_Y = \pi^*(K_X+S+B+\Lambda^m)$ is Cartier up to multiplying by $p^e-1$ for a sufficiently divisible $e$, the second identity follows from the first one by the projection formula.
\end{proof}

The following lemma allows for lifting sections.
\begin{lemma} \label{lem:integral_s0_adjunction} With notation as in Setting \ref{setting:s0_integral_adjunction} suppose that $(X,S+B)$ is plt with standard coefficients, $S$ is an irreducible divisor, and $A := -(K_X+S+B)$ is ample. Assume that $\Supp \Lambda$ contains the non-snc locus of $(X,S+B)$ and write $A_{\tilde S} := -(K_{\tilde S} + B_{\tilde S}) = -(K_X+S+B)|_{\tilde S}$ for the normalisation $\tilde S$ of $S$.  Then by restricting sections we get a surjection
\[
C^0_{\Lambda}(X, S+B; \lfloor mA \rfloor) \to C^0_{\Lambda_{\tilde S}}(\tilde S, B_{\tilde S}; \lfloor mA_{\tilde S} \rfloor),
\]
where $\Lambda_{\tilde S} := \Lambda|_{\tilde S}$.
\end{lemma}
\begin{proof}
Let $\pi \colon Y \to X$ be a log resolution of $(X,S+B)$ which is an isomorphism over the simple normal crossings locus. We can write 
\begin{align*}
K_Y+S_Y+B_Y &= \pi^*(K_X+S+B), \text{ and } \\
K_{S_Y} + B_{S_Y} &= (K_Y+S_Y+B_Y)|_{S_Y}.
\end{align*}
for $S_Y := \pi^{-1}_*S$. Define $L_Y$, $L_{S_Y}$, $\Phi_Y$, $\Phi_{S_Y}$ as in Setting \ref{setting:s0_integral_adjunction}. 

Pick a $\pi$-exceptional effective anti-ample divisor $E$. Let 
\[
\Lambda_Y := \pi^*\Lambda + \epsilon E
\]
for $0 < \epsilon \ll 1$ such that $\Lambda_Y$ satisfies the assumptions of Setting \ref{setting:s0_integral_adjunction} and $\Supp \Lambda_Y = \Supp \pi^*\Lambda$. Set $\Lambda_{S_Y} := \Lambda_Y|_{S_Y}$. 

By the standard adjunction for $S^0$ (see e.g.\ \cite[Proposition 2.3]{hx13}),  since \[ L_Y-(K_Y+\Phi _Y+\Lambda _Y)=-(m+1)(K_Y+S_Y+B_Y+\Lambda _Y ^m)\] is ample, restricting sections induces a surjective map
\[
S^0(Y, \Phi_Y + \Lambda_Y; L_Y) \to S^0(S_Y, \Phi_{S_Y} + \Lambda_{S_Y}; L_{S_Y}).
\]
Indeed, $K_{S_Y} + \Phi_{S_Y} = (K_Y + \Phi_Y)|_{S_Y}$ and $L_Y|_{S_Y} = L_{S_Y}$ as $(Y,S_Y + B_Y)$ is log smooth. Thus,
$
C^0_{\Lambda _Y}(Y, S_Y + B_Y; L_Y) \to C^0_{\Lambda _{S_Y}}(S_Y, B_{S_Y}; L_{S_Y})
$
is surjective, and the claim follows from Lemma \ref{lem:integral_s0_birational} applied to both sides.
\end{proof}

Finally, we show that $C^0$ gets smaller when the boundary gets bigger. 
\begin{lemma} \label{lem:integral_s0_grows} Let $(X,S+B)$ and $(X,S'+B')$ be two sub-log pairs satisfying the assumptions of Setting \ref{setting:s0_integral_adjunction}. Suppose that  $S'+B' \geq S+B$ and define $L$ and $L'$ for $(X,S+B)$ and $(X,S'+B')$, respectively, as in Setting \ref{setting:s0_integral_adjunction}. 

Then $L-L' \geq 0$ and the inclusion $H^0(X,\OO _X(L')) \subseteq H^0(X,\OO _X(L))$ induces an inclusion
\[
C^0_{\Lambda'}(X,S'+B'; L') \subseteq C^0_{\Lambda}(X, S+B; L),
\]  
where $\Lambda$, $\Lambda'$ are as in Setting \ref{setting:s0_integral_adjunction} and $\Supp \Lambda \subseteq \Supp \Lambda' \cup (S'-S)$.
\end{lemma}
Note that it would be too restrictive to assume that $\Supp \Lambda \subseteq \Supp \Lambda'$. Indeed, $\Lambda'$ as in Setting \ref{setting:s0_integral_adjunction} has no common components with $S'$, while $\Lambda$ has no common components with $S$, but will often have common components with $S'-S$.  
\begin{proof}
Let $\Phi$ and $\Phi'$ be defined for $(X,S+B)$ and $(X,S'+B')$ as in Setting \ref{setting:s0_integral_adjunction}. By (\ref{eq:s0_integral}), we have
\begin{align*}
L - L' &= \Phi - \Phi' + (m+1)(S'+B'-S-B) \\
		&= S-S'+\lfloor (m+1)(S'+B') \rfloor - \lfloor (m+1)(S+B) \rfloor,
\end{align*}
and so $L - L' \geq 0$.

 We may assume that $\Lambda \leq \Lambda' + (m+1)(S'+B'-S-B)$. Then
\begin{align*}
S^0(X, \Phi'+\Lambda'; L') &\subseteq S^0(X, \Phi' + \Lambda' + (L - L'); L) \\
				  &= S^0(X, \Phi + \Lambda' + (m+1)(S'+B'-S-B); L) \\
				  &\subseteq S^0(X, \Phi+\Lambda; L). \qedhere
\end{align*}
\end{proof}

\subsection{Proof of Theorem \ref{thm:lift_complements}}
We are ready to show that $m$-complements of pl-flipping contractions exist for $m \in \{1,2,3,4,6\}$. With notation as in Setting \ref{setting:s0_integral_adjunction}, note that 
\[
\lfloor mA \rfloor = \lfloor -m(K_X + S+B) \rfloor = -m(K_X+S+B^*)
\]
for $B^* := \frac{1}{m} \lceil mB \rceil \geq B$. When $B$ has standard coefficients, then the defect is zero, $B^* = \frac{1}{m} \lfloor (m+1)B \rfloor$, and $L = \lfloor mA \rfloor = -m(K_X+S+B^*)$.

\begin{proof}[Proof of Theorem \ref{thm:lift_complements}]

We may assume that $Z$ is affine. 
Let $\tilde S$ be the normalisation of $S$. By Lemma \ref{lem:integral_s0_adjunction}, restricting sections gives a surjective map
\[
C^0_{\Lambda}(X, S+B; -m(K_X+S+B^*)) \to C^0_{\Lambda_{\tilde S}}(\tilde S, B_{\tilde S}; -m(K_{\tilde S}+B^*_{\tilde S})),
\] 
where $K_{\tilde S} + B_{\tilde S} = (K_X+S+B)|_{\tilde S}$, $B^*_{\tilde S} = \frac{1}{m}\lceil mB_{\tilde S} \rceil$, and $\Lambda$ is as in Setting \ref{setting:s0_integral_adjunction} with $\Supp \Lambda$ containing $\mathrm{Exc}(f)$ and the non-snc locus of $(X,S+B)$. Set $\Lambda_{\tilde S} := \Lambda|_{\tilde S}$, and note that it satisfies the assumptions of Setting \ref{setting:s0_integral_adjunction} for $(\tilde S, B_{\tilde S})$.

By Lemma \ref{lem:s0_complements_surfaces}, there exists $ \Gamma_{\tilde S} \in |{-}m(K_{\tilde S}+B^*_{\tilde S})| $ such that $(\tilde S, B^c_{\tilde S})$ is an $m$-complement of $(\tilde S, B_{\tilde S})$ for $B^c_{\tilde S} = B^*_{\tilde S} + \frac{1}{m}\Gamma_{\tilde S}$, and which moreover lifts to
\[
\Gamma \in |{-}m(K_X+S+B^*)|.
\]
Set $B^c = B^* + \frac{1}{m}\Gamma$. Then $m(K_X + S + B^c) \sim 0$ and $(K_X + S + B^c)|_{\tilde S} = K_{\tilde S} + B^c_{\tilde S}$. By inversion of adjunction (\cite[Corollary 1.5]{HW18}) applied to $(X,S + (1-\epsilon)B^c)$ for $0 < \epsilon \ll 1$, we get that $(X,S+B^c)$ is lc in a neighbourhood of ${\rm Exc}\, f$, and hence it is an $m$-complement of $(X,S+B)$.
\end{proof} 

In the above proof we used the following lemma.
\begin{lemma} \label{lem:s0_complements_surfaces} Let $(X,B)$ be a two dimensional klt pair with standard coefficients defined over a perfect field of characteristic $p>2$ and let $f \colon X \to Z$ be a projective birational map such that $-(K_X+B)$ is ample. Then there exists $m \in \{1,2,3,4,6\}$ and
\[
s \in C^0_{\Lambda}(X, B; -m(K_X+B^*)) \subseteq H^0(X, -m(K_X+B^*)),
\] 
such that $(X, B^* + \frac{1}{m}\Gamma)$ is an $m$-complement of $(X,B)$ in a neighbourhood of $\mathrm{Exc}(f)$ where $B^* := \frac{1}{m}\lceil mB \rceil$ and $\Gamma$ is the divisor corresponding to $s$. Here $\Lambda$ is as in Setting \ref{setting:s0_integral_adjunction}.
\end{lemma}
\begin{proof}

By Lemma \ref{lemma:surface_mmp_lemma}, there exist an irreducible, exceptional over $Z$, curve $C$ on a blow-up of $X$ and projective birational maps $g \colon Y \to X$ and $h \colon Y \to W$ over $Z$ such that
\begin{enumerate}
	\item $g$ extracts $C$ or is the identity if $C \subseteq X$,
	\item $(Y, C+ B_Y)$ is plt,
	\item $(W, C_W + B_W)$ is plt and $-(K_{W}+C_W+B_W)$ is ample over $Z$,
	\item $B^+_Y - B_Y \geq 0$,
\end{enumerate} 
where $K_{Y}+bC + B_Y = g^*(K_X+B)$ for $C \not \subseteq \Supp B_Y$, $C_W := h_*C \neq 0$, $B_W := h_*B_Y$, and $K_Y + C + B^+_Y = h^*(K_W+C_W+B_W)$.

We have 
\begin{align*}
C_{\Lambda}^0(X, B; L) &= C^0_{\Lambda_Y}(Y, bC+B_Y; L_Y)\\
			 &\supseteq C^0_{\Lambda^+_Y}(Y, C + B^+_Y; L^+_Y) \\
			 &= C^0_{\Lambda_W}(W, C_W + B_W; L_W),
\end{align*}
where $L$, $L_Y$, $L^+_Y$, and $L_W$ are defined as in Setting \ref{setting:s0_integral_adjunction}  and the defects $D$ and $D_W$ vanish as $B$ and $C_W+B_W$ have standard coefficients.
 The first and the third equality hold by Lemma \ref{lem:integral_s0_birational}  and the middle inclusion holds by Lemma \ref{lem:integral_s0_grows}, since $C+B^+_Y\geq bC+B_Y$. Here, the perturbation divisors were chosen in the following way. First, we set $\Lambda_Y := g^*\Lambda$. Second we pick $\Lambda_W$ for $(W,C_W+B_W)$ as in Setting \ref{setting:s0_integral_adjunction}. By the construction in Remark \ref{rem:integral_s0_perturbations2}, we can assume that $\Lambda_W$ contains $g(\Supp (\Lambda_Y-\Lambda_Y \wedge C) \cup \mathrm{Exc}(h))$ and the non-snc locus of $(W,C_W+B_W)$ in its support. Last, set $\Lambda^+_Y := h^*\Lambda_W$.

Note that $L = -m(K_X+B^*)$ and $L_W = -m(K_W+ C_W + B^*_W)$ for $B^*_W = \frac{1}{m} \lceil mB_W \rceil$. By Lemma \ref{lem:integral_s0_adjunction}, restricting sections thus gives a surjective map
\[
C^0_{\Lambda_W}(W, C_W+B_W; -m(K_W + C_W + B^*_W)) \to C^0_{\Lambda_C}(C, B_C; -m(K_C+B_C^*)),
\]
where $C$ is identified with $C_W$ and $K_C + B_C = (K_W + C_W + B_W)|_C$. As usual $B_C^* := \frac{1}{m}\lceil mB_C \rceil$ and $\Lambda_C := \Lambda_W|_{C}$.\\

Let $m \in \{1,2,3,4,6\}$ be the minimal number such that $(C, B_C)$ admits an $m$-complement. By Lemma \ref{lem:gfr_p1_complements}, $(C,\{(m+1)B_C\})$ is globally F-regular, and so
\[
C^0_{\Lambda_C}(C, B_C; -m(K_C+B_C^*)) = H^0(C, -m(K_C + B^*_C)).
\]
In particular, there exists an lc $m$-complement $(C,B^c_C)$ of $(C,B_C)$ for some $m \in \{1,2,3,4,6\}$ (and hence of $(C,B^*_C)$ as $mB^*_C = \lceil mB_C \rceil$) which can be lifted to $W$. More precisely, there exists a non-zero section 
\[
s \in C^0_{\Lambda_W}(W, C_W+B_W; -m(K_W + C_W + B^*_W))
\]
with associated divisor $\Gamma$ such that $m(K_W + C_W + B^c_W) \sim 0$ and 
\[
(K_W + C_W + B^c_W)|_C = K_C + B^c_C,
\]
where $B^c_W := B^*_W + \frac{1}{m}\Gamma$. By inversion of adjunction, $(W,C_W+B^c_W)$ is log canonical along $C_W$. Note that 
\[
K_W + C_W + \epsilon B_W + (1-\epsilon)B^c_W
\]
is thus plt along $C_W$ and $\Q$-equivalent over $Z$ to $\epsilon(K_W+C_W+B_W)$, and hence by the Koll\'ar-Shokurov connectedness principle (cf.\ \cite[Theorem 5.2]{tanaka16_excellent}), it is plt for any $0 < \epsilon < 1$. Hence $(W,C_W+B^c_W)$ is lc, and thus an $m$-complement of $(W,C_W + B_W)$.

Let $K_Y + C + B^c_Y = h^*(K_W + C_W + B^c_W)$ and $B^c := g_*(C + B^c_Y)$. Then $(X,B^c)$ is an $m$-complement of $(X,B)$ which by the above inclusions of $C^0$ corresponds to a section in $C^0_{\Lambda}(X, B; -m(K_X+B^*))$.
\end{proof}

\begin{remark} \label{rem:lift_for_non_f_regular} With notation as in Theorem \ref{thm:lift_complements}, if $(X,S+B)$ is not purely relatively F-regular and $p=5$, then $m=6$. Indeed, under these assumptions, $(\tilde S, B_{\tilde S})$ is not relatively F-regular by F-adjunction, and hence, in the proof of Lemma \ref{lem:s0_complements_surfaces}, we have that  $B_C = \frac{1}{2}P_1 + \frac{2}{3}P_2 + \frac{4}{5}P_3$ for distinct points $P_1$, $P_2$, and $P_3$, by Remark \ref{rem:surfaces_char_5}. The smallest $m$ for which this $(C,B_C)$ admits an $m$-complement is $m=6$.
\end{remark}

\begin{lemma} \label{lem:gfr_p1_complements} Let $(\mathbb{P}^1, B)$ be a log pair with standard coefficients and $\deg B < 2$ defined over a perfect field of characteristic $p>2$. Let $m \in \{1,2,3,4,6\}$ be the minimal number such that $(\mathbb{P}^1, B)$ admits an $m$-complement. Then $(\mathbb{P}^1, \{(m+1)B\})$ is globally F-regular.
\end{lemma}
\begin{proof}
If $B$ is supported at two or fewer points, then so is $\{(m+1)B\}$, and hence $(\mathbb{P}^1, \{(m+1)B\})$ is globally F-regular. Indeed, one can always increase one of the coefficients to one and apply global F-adjunction.

Thus, we can assume that $B = a_1P_1 + a_2P_2 + a_3P_3$ for distinct points $P_1$, $P_2$, $P_3$ and 
$
(a_1,a_2,a_3) \in \{(\frac{1}{2}, \frac{1}{2}, 1 - \frac{1}{n}), (\frac{1}{2}, \frac{2}{3}, \frac{2}{3}), (\frac{1}{2}, \frac{2}{3}, \frac{3}{4}), (\frac{1}{2}, \frac{2}{3}, \frac{4}{5}) \},
$
where $n \in \mathbb N$ is arbitrary. These are two, three, four, and six complimentary, respectively. 

Therefore $\{(m+1)B\} = b_1P_1 + b_2P_2 + b_3P_3$ for
\[ (b_1,b_2,b_3) \in \{ (\frac{1}{2}, \frac{1}{2}, \frac{1}{2}), (\frac{1}{2}, \frac{1}{2}, \frac{n-3}{n}), (0, \frac{2}{3}, \frac{2}{3}), (\frac{1}{2}, \frac{1}{3}, \frac{3}{4}), (\frac{1}{2}, \frac{2}{3}, \frac{3}{5}) \},\]
where $n\geq 3$.

To solve the first two cases, it is enough to show that $(\mathbb{P}^1, \frac{1}{2}P_1 + \frac{1}{2}P_2 + (1-\frac{1}{n})P_3)$ is globally F-regular which follows by \cite[Theorem 4.2]{watanabe91}. For the next two cases, we can argue as in the first paragraph: by increasing the biggest coefficient to one (obtaining $(0, \frac{2}{3}, 1)$, $(\frac{1}{2}, \frac{1}{3}, 1)$) and applying F-adjunction. When $p\geq 5$, the last case follows by increasing $\frac{3}{5}$ to $\frac{3}{4}$ and applying \cite[Theorem 4.2]{watanabe91} again.

We are left to show the last case for $p=3$. By Fedder's criterion, it is enough to check that $(x+y)^{c_1}x^{c_2}y^{c_3}$ contains a monomial $x^iy^j$ for some $i,j < p^e-1$ and $e>0$, where $c_r := \lceil (p^e-1)b_i \rceil$ and $r \in \{1,2,3\}$. Take $e=3$. Then, we have $(c_1,c_2,c_3) = (13, 18, 16)$ and
\[
(x+y)^{13}x^{18}y^{16} = \ldots + {13 \choose 9} x^{22}y^{25} + \ldots, 
\]
where $3$ does not divide ${13 \choose 9} = \frac{10 \cdot 11 \cdot 12 \cdot 13}{4!}$.
\end{proof}

\section{Flips admitting a qdlt complement}
The goal of this section is to show the existence of flips for flipping contractions admitting a qdlt $k$-complement, where $k\in \{1,2,3,4,6\}$.
\begin{proposition}[{cf.\ Proposition \ref{proposition:building_blocks}}]  \label{proposition:building_blocks_qdlt} Let $(X,\Delta)$ be a $\Q$-factorial qdlt three-dimensional pair with standard coefficients over a perfect field $k$ of characteristic $p>3$. Let $f \colon X\to Z$ be a $(K_X+\Delta)$-flipping contraction such that $\rho(X/Z)=1$ and let $\Sigma$ be a flipping curve. Assume that there exists a qdlt $6$-complement $(X,\Delta^c)$ of $(X,\Delta)$ such that $\Sigma \cdot S < 0$ for some irreducible divisor $S \subseteq \lfloor \Delta^c \rfloor$. Then the flip $f^+ \colon X^+\to Z$ exists.
\end{proposition} 
\begin{proof}
Write $\Delta = aS + D + B$, where $1\geq a \geq 0$, the divisor $D$ is integral, $S \not \subseteq \Supp(D+B)$, and $\lfloor B \rfloor = 0$. By replacing $\Delta$ by $S + (1-\frac{1}{k})D + B$ for $k \gg 0$, we can assume that $(X,\Delta)$ is plt. As explained in the introduction we split the proof into three cases.
\begin{enumerate}
\item $(X,\Delta^c)$ is plt along the flipping locus, or
\item $\Sigma \cdot E < 0$ for a divisor $E \subseteq \lfloor \Delta^c \rfloor$ different from $S$, or
\item $\Sigma \cdot E \geq 0$ for a divisor $E \subseteq \lfloor \Delta^c \rfloor$ intersecting the flipping locus.
\end{enumerate} 

 Case (1) and Case (3) follow from Proposition \ref{proposition:flip_case1} and Proposition \ref{proposition:flip_case3}, respectively, applied to $(X,\Delta)$. Case (2) follows from Proposition \ref{proposition:flip_case2} applied to $(X,\Delta +bE)$ where $b\geq 0$ is such that ${\rm mult}_{E}(\Delta + bE)=1$.
\end{proof}

\begin{proposition} \label{proposition:flip_case1} Let $(X,S+B)$ be a three-dimensional $\Q$-factorial plt pair over a perfect field $k$ of characteristic $p>3$ with $S$ irreducible and $B$ having standard coefficients. Let $f \colon X \to Z$ be a pl-flipping contraction such that $\rho(X/Z)=1$. Assume that there exists a plt $6$-complement $(X,S+B^c)$ of $(X,S+B)$ over $Z$. Then, the flip exists.
\end{proposition}
\begin{proof}
Write $K_{\tilde S} + B_{\tilde S} = (K_X+S+B)|_{\tilde S}$ and $K_{\tilde S} + B^c_{\tilde S} = (K_X+S+B^c)|_{\tilde S}$ for the normalisation $\tilde S$ of $S$. The pair $(\tilde S, B^c_{\tilde S})$ is a klt $6$-complement, and so $(\tilde S, B_{\tilde S})$ {is relatively F-regular} by Proposition \ref{proposition:no_klt_complements}. In particular, the flip exists by  Theorem \ref{thm:hx_flips}. 
\end{proof}

The following proposition addresses Case (2).
\begin{proposition} \label{proposition:flip_case2} Let $(X,\Delta)$ be a three-dimensional $\Q$-factorial qdlt pair  over a perfect field $k$ of characteristic $p>0$, let $f \colon X \to Z$ be a flipping contraction such that $\rho(X/Z)=1$, and let $\Sigma$ be a flipping curve. Assume that there exist distinct irreducible divisors $S, E \subseteq \lfloor \Delta \rfloor$ such that $S \cdot \Sigma < 0$ and $E \cdot \Sigma < 0$. Then the flip of $\Sigma$ exists.
\end{proposition}
\begin{proof}
We may assume that $Z$ is a sufficiently small affine neighbourhood of $Q := f(\Sigma)$. Let $k,l \in \N$ be such that $kS \sim _Z lE$ are Cartier and consider a pencil $h \colon X \dashrightarrow \mathbb{P}^1_Z$ given by the linear system in $|kS|$ induced by these two divisors. We set $X'$ to be the closure of the image of $X$ under $h$.

Since $(X,S+E)$ is qdlt and $\mathrm{Exc}(f) \subseteq S \cap E$, we get that $S \cap E = \mathrm{Exc}(f)$. Thus the induced map $g \colon X' \to Z$ is an isomorphism over $Z\, \backslash\, Q$, and $g$ is a small birational morphism. If $S'$ is the strict transform of $S$, then $kS'$ is the restriction of a section of $\mathbb{P}^1_Z$, and so $S'$ is $\Q$-Cartier and relatively ample. Let $\pi \colon X^+ \to X'$ be the normalisation of $X'$. Then, $X\dasharrow X^+$ is a small birational morphism of normal varieties, and we have that
\[
\bigoplus_{m \in \mathbb{Z}_{\geq 0}} H^0(X, mS) = \bigoplus_{m \in \mathbb{Z}_{\geq 0}} H^0(X^+, m\pi^*S')
\] 
is finitely generated. Since $K_X+\Delta \sim_{\Q} aS$ for $a \in \Q_{>0}$, the flip of $X$ exists by \cite[Corollary 6.4]{km98}. 
\end{proof}

Now, we deal with Case (3). Note that we will apply this proposition later in the case when $B$ does not have standard coefficients.
\begin{proposition} \label{proposition:flip_case3} 
Let $(X,S+B)$ be a three-dimensional $\Q$-factorial qdlt pair over a perfect field $k$ of characteristic $p>3$ with $S$ irreducible, let $f \colon X \to Z$ be a flipping contraction such that $\rho(X/Z)=1$, $-(K_X+S+B)$ is relatively ample, and $-S$ is relatively ample. Let $\Sigma$ be a flipping curve. Assume that there exists a $6$-complement $(X,S+E+B^c)$ of $(X,S+B)$ such that $E$ is irreducible, $E \cdot \Sigma \geq 0$, and $E \cap \Sigma \neq \emptyset$. Then, the flip of $\Sigma$ exists.

\end{proposition}

\begin{proof}
Let $\tilde S$ be the normalisation of $S$. By perturbing the coefficients of $\lfloor B\rfloor $, we may assume that $(X,S+B)$ is plt. The pair $(\tilde S, B_{\tilde S})$ admits a $6$-complement $(\tilde S, E|_{\tilde S} + B^c_{\tilde S})$, where $K_{\tilde S} + B_{\tilde S} = (K_X+S+B)|_{\tilde S}$ and $K_{\tilde S} + E|_{\tilde S} + B^c_{\tilde S} = (K_X+S+E+B^c)|_{\tilde S}$.  

 We claim that $E|_{\tilde S}$ is not exceptional over $Z$. Indeed, otherwise
 \[
 	0 > (E|_{\tilde S})^2 =  E \cdot (E\cap S) =  E \cdot \sum \lambda _i \Sigma _i \geq 0, 
 \]
 for some flipping curves $\Sigma _i$ and some $\lambda_i > 0$, which is a contradiction. We have used the fact that as $\rho (X/Z)=1$, if $E\cdot \Sigma \geq 0$, then $E\cdot \Sigma _i\geq 0$ for every flipping curve $\Sigma _i$.

By Lemma \ref{lem:surface_char_5_with_special_complement}, the pair $(\tilde S, B_{\tilde S})$ is relatively F-regular over a neighbourhood of $f(\Sigma)$ in $f(S)$, and so the flip exists by Theorem \ref{thm:hx_flips}.

\end{proof}

\section{Divisorial extractions}

In \cite{HW18}, we have shown that dlt modifications exist. In our proof of the existence of flips, it is important to construct minimal qdlt modifications of flipping contractions. To this end, we need to extract a single divisorial place, and the following proposition shows that this can be done for $6$-complements.
\begin{proposition} \label{prop:single_extraction} Let $(X,\Delta)$ be a $\mathbb Q$-factorial three-dimensional lc pair defined over a perfect field of characteristic $p>3$ such that  $X$ is klt and $6(K_X + \Delta) \sim 0$. Let $E$ be a non-klt valuation of $(X,\Delta)$ over $X$. Then there exists a projective birational morphism $g \colon Y \to X$ such that $E$ is its exceptional locus.
\end{proposition}
\begin{proof}
Let $\pi \colon Y \to X$ be a dlt modification of $(X,\Delta)$ such that $E$ is a divisor on $Y$ (see \cite[Corollary 1.4]{HW18}). Let $\mathrm{Exc}(\pi) = E + E_1 + \ldots  + E_m$. Write
\begin{align*}
K_Y &+ \Delta_Y = \pi^*(K_X+\Delta) \\
K_Y &+ (1-\epsilon)\pi^{-1}_*\Delta + aE + a_1E_1 + \ldots + a_mE_m = \pi^*(K_X + (1-\epsilon)\Delta),
\end{align*}
where $a,a_1,\ldots, a_m < 1$ as $X$ is klt, and set
\[
\Delta' = (1-\epsilon)\pi^{-1}_*\Delta + aE + E_1 + \ldots + E_m. 
\]
By taking $0 < \epsilon \ll 1$, we can assume that $a>0$. Note that
\begin{equation} \label{eq:mmp_extraction}
K_Y + \Delta' \sim_{\Q, X} (1-a_1)E_1 + \ldots + (1-a_m)E_m,
\end{equation} {so that the $(K_Y + \Delta')$-MMP over $X$ will not contract $E$ and the contracted loci are always contained in the support of the strict transform of $(1-a_1)E_1 + \ldots + (1-a_m)E_m$.}
The negativity lemma implies that the output of a $(K_Y + \Delta')$-MMP over $X$ is the sought-for extraction of $E$. Hence, it is enough to show that we can run such an MMP.

By induction, we can assume that we have constructed the $n$-th step of the MMP $h \colon Y \dashrightarrow Y_n$ and we need to show that we can construct the $(n+1)$-st step. Let $\pi_n \colon Y_n \to X$ be the induced morphism, $\Delta'_{n} := h_* \Delta'$, and $\Delta_{n} = h_*\Delta_{Y}$. By abuse of notation, we denote the strict transforms of $E,E_1,\ldots, E_m$ by the same symbols. 

The cone theorem is valid by \cite{keel99} (cf.\ \cite[Theorem 2.4]{HW18}). Let $R$ be a $K_{Y_n} + \Delta'_n$ negative extremal ray. By (\ref{eq:mmp_extraction}), we have that $R \cdot E_i < 0$ for some $i \geq 1$. Thus the contraction $f \colon Y_n \to Y'_n$ of $R$ exists by \cite[Theorem 1.2 and Proposition 2.6]{HW18}. 

If $f$ is divisorial, then we set $Y_{n+1} := Y'_n$. If $f$ is a flipping contraction, then the proof of \cite[Lemma 3.1]{HW18} applied to $(Y_n, \Delta_{n})$ over $X$ implies the existence of a  divisor $E' \subseteq \mathrm{Exc}(\pi_n)$ such that $R \cdot E' > 0$. 
Since $(Y_n, \Delta_n')$ is dlt, $(Y_n, \Delta_n)$ is lc, $6(K_{Y_n} + \Delta_{n}) \sim_{\pi_n} 0$, and $E' \leq \Delta_n$, we can apply Proposition \ref{proposition:flip_case3} to infer the existence of the flip of $f$. 

The termination of this MMP follows by the usual special termination argument.
\end{proof}

Let $(X,S+B)$ be a three-dimensional plt pair with different $B_S$ and let $(X,S+B^c)$ be a $k$-complement with different $B^c_S$. Assume for simplicity that $S$ is normal.  Then $(S,B^c_S)$ is a $k$-complement of $(S,B_S)$. Assume that $(S,B^c_S)$ admits a unique non-klt place, i.e.\ it has a dlt modification with an irreducible exceptional curve. Such complements are of fundamental importance in this article due to Proposition \ref{proposition:no_klt_complements}. By inversion of adjunction, $(X,S+B^c)$ has a unique log canonical centre strictly contained in $S$ but infinitely many log canonical places over this center. Thus, its dlt modifications might be very complicated with many exceptional divisors. The following corollary shows that this problem may be solved by allowing qdlt singularities: under the above assumptions it stipulates that there exists a qdlt modification with an irreducible exceptional divisor.
\begin{corollary} \label{cor:qdlt_modification} Let $(X,S+B)$ be a $\mathbb Q$-factorial three-dimensional plt pair defined over a perfect field of characteristic $p>3$ where $X$ is klt and  $S$ is a prime divisor. Assume that $(X,S+B)$ admits a $6$-complement $(X,S+B^c)$ such that   $(\tilde S, B^c_{\tilde S})$ has a unique non-klt place, where $K_{\tilde S} + B^c_{\tilde S} = (K_X+S+B^c)|_{\tilde S}$ and $\tilde S$ is the normalisation of $S$. 

Then $(X,S+B^c)$ is qdlt in a neighbourhood of $S$, or $\lfloor B^c \rfloor$ is disjoint from $S$ and there exists a projective birational map $\pi \colon Y \to X$ such that $(Y,S_Y+B^c_Y)$ is qdlt over a neighbourhood of $S$, the exceptional divisor $E$ is irreducible {and $E \subseteq \lfloor B_Y^c \rfloor$}, where $K_Y + S_Y + B_Y^c = \pi^*(K_X+S+B^c)$.
\end{corollary} 
 In particular, this corollary implies that if $(X,S+B^c)$ is not qdlt, then the log canonical centres in a neighbourhood of $S$ are the generic points of $\pi(S_Y \cap E)$, $\pi(E)$, and $S$ itself. Note that $S_Y \cap E$ must be irreducible as $(\tilde S, B^c_{\tilde S})$ has a unique log canonical place. Now there are two possibilities: either $\pi(E)\subseteq S$ in which case $(X,S+B)$ admits a unique log canonical centre $\pi(E)=\pi(S_Y\cap E)$ (a point or a curve), or $\pi(E) \not \subseteq S$ is a curve intersecting $S$ at the point $\pi(S_Y \cap E)$. Moreover, if $(X,S+B^c)$ is qdlt, then the proof below shows that $\lfloor B^c \rfloor$ is irreducible in a neighbourhood of $S$ and intersects $S$ at its unique non-klt place (which is a curve).
\begin{proof}
We work in a sufficiently small open neighbourhood of $S$. First, suppose that $\lfloor B^c \rfloor$ is non-empty and intersects $S$. Under this assumption the unique log canonical centre of $(\tilde S,B^c_{\tilde S})$ must be an irreducible curve given as $\lfloor B^c \rfloor|_{\tilde S}$. In particular,  $\lfloor B^c \rfloor$ is irreducible (cf.\ Remark \ref{remark:generic_point_of_stratum}), the pair $(\tilde S, B^c_{\tilde S})$ is plt, and $(X,S+B^c)$ is qdlt by Lemma \ref{lem:qdlt_inv_adjunction}.

Thus, we can assume that $\lfloor B^c \rfloor = 0$, and so the dlt modification $\pi \colon Y\to X$ is nontrivial. Set $K_Y + \Delta^c_Y = \pi^*(K_X+S+B^c)$ and pick an irreducible exceptional divisor $E_1$ which is not an articulation point of $D(\Delta_Y^{c,=1})$ (for example pick any divisor with the farthest distance edge-wise in $D(\Delta_Y^{c,=1})$ from the node corresponding to $S$). Let $g\colon X_1 \to X$ be the extraction of $E_1$ (see Proposition \ref{prop:single_extraction}) and write 
\[
K_{X_1}+S_1+E_1+B_1^c=g^*(K_X+S+B^c)
\]
where $S_1,B^c_1$ are the strict transforms of $S,B^c$, respectively. {Note that $S_1$ intersects $E_1$.}

We claim that $(X_1,S_1+E_1+B_1^c)$ is qdlt in a neighbourhood of $S_1$. To this end we note that 
\[
K_{\tilde S_1} + B_{\tilde S_1}^c := (K_{X_1} + S_1 + E_1 + B_1^c)|_{\tilde S_1} = g^*(K_{\tilde S} + B^c_{\tilde S}),
\]
where $\tilde S_1$ is the normalisation of $S_1$. Since $(\tilde S, B^c_{\tilde S})$ admits a unique non-klt place, we get that $(\tilde S_1, B_{\tilde S_1}^c)$ is plt. In particular, Lemma \ref{lem:qdlt_inv_adjunction} implies that $(X_1,S_1+E_1+B_1^c)$ is qdlt in a neighbourhood of $S_1$. 

Therefore, it is enough to show that $(X_1,S_1+E_1+B_1^c)$ does not admit a log canonical centre which is disjoint from $S_1$ and intersects $E_1$. By contradiction, assume that it does admit such a log canonical centre. Let $h \colon W \to X_1$ be a projective birational morphism which factors through $Y$ 
\[
g \circ h \colon W \xrightarrow{h_Y} Y \xrightarrow{\pi} X,
\]
and  such that $g \circ h$ is a log resolution of $(X,S+B)$. Write $K_W + \Delta^c_W= h^*(K_{X_1}+S_1+E_1+B_1^c)$. Since $S_1 \cap E_1$ is disjoint from the other log canonical centres, the strict transform $E_{W,1}$ of $E_1$ is an articulation point of $D({\Delta_W^{c,=1}})$. Since $K_W + \Delta^c_W = h_Y^*(K_Y + \Delta^c_Y)$, Lemma \ref{lemma:articulation_points} implies that $E_1$ is an articulation point of $D(\Delta_Y^{c,=1})$ which is a contradiction. In particular, $S_1$, $E_1$, and the irreducible curve $S_1 \cap E_1$ are the only log canonical centres of $(X_1,S_1+E_1+B_1^c)$.
\end{proof}

\section{Existence of flips}
In this section we prove the main theorem. We start by showing the following result.
\begin{theorem} \label{theorem:almost_flips} Let $(X,\Delta)$ be a $\Q$-factorial klt three-dimensional pair with standard coefficients defined over a perfect field $k$ of characteristic $p=5$. If $f \colon X\to Z$ is a flipping contraction, then the flip $f^+ \colon X^+\to Z$ exists.\end{theorem}
\begin{proof}
 We will assume throughout that $Z$ is a sufficiently small affine neighbourhood of $Q := f(\mathrm{Exc}(f))$.\ We say that a $\Q$-Cartier divisor $D$ is ample if it is relatively ample over $Z$.

By Shokurov's reduction to pl-flips, it suffices to show the existence of pl-flips. Let $(X,S+B)$ be a plt pair with standard coefficients and $f \colon X\to Z$ a pl-flipping contraction. In particular $-S$ and $-(K_X+S+B)$ are $f$-ample,  and so $\mathrm{Exc}(f) \subseteq S$.  By Theorem \ref{thm:hx_flips}, the flip exists unless $(\tilde S,B_{\tilde S})$ is not globally F-regular over $T=f(S)$ where $K_{\tilde S}+B_{\tilde S}=(K_X+S+B)|_{\tilde S}$ and $\tilde S$ is the normalisation of $S$. Thus, we can assume that $(\tilde S,B_{\tilde S})$ is not globally F-regular over $T$. 

 Theorem \ref{thm:lift_complements} shows the existence of an $m$-complement $(X,S+B^c)$ of $(X,S+B)$. Since $(X,S+B)$ is not relatively purely F-regular, Remark \ref{rem:lift_for_non_f_regular} implies that $m=6$.  Let $(\tilde S, B^c_{\tilde S})$ be the induced 6-complement of $(\tilde S,B_{\tilde S})$.  By Proposition \ref{proposition:no_klt_complements}, the pair $(\tilde S,B^c_{\tilde S})$ has a unique place $C$ of log discrepancy zero which is exceptional over $T$.

If $(X,S+B^c)$ is qdlt, then the flip exists by Proposition \ref{proposition:building_blocks_qdlt}. Thus, by Corollary \ref{cor:qdlt_modification}, we may assume that $\lfloor B^c \rfloor = 0$ and there exists a qdlt modification $g \colon X_1 \to X$  of $(X, S + B^c)$ with an irreducible exceptional divisor $E_1$. Let $f_1 \colon X_1 \to Z$ be the induced map to $Z$, and write $K_{X_1} + S_1 + B_1 + aE_1 = g^*(K_X+S+B)$, and $K_{X_1} + S_1 + B^c_1 + E_1 = g^*(K_X+S+B^c)$. In particular, $S_1 \cap E_1$ is the unique log canonical place of $(\tilde S, B_{\tilde S})$, and so there are two possibilities: either $g(E_1)\subseteq S$ and $f_1(E_1)=Q$, or $g(E_1) \not \subseteq S$ is a curve intersecting $S$.

We would like to run a $(K_{X_1}+S_1+B_1+aE_1)$-MMP. It could possibly happen that $a<0$ so we take $0< \lambda \ll 1$ and set
\[
\Delta_1 := \lambda(S_1 + B_1 + aE_1) + (1-\lambda)(S_1 + E_1 + B_1^c),
\]
so that $K_{X_1} + \Delta_1 \sim_{\mathbb{Q}, Z} \lambda(K_{X_1} + S_1 + B_1 + aE_1)$, and $(X_1, \Delta_1)$ is plt. 

Since $\rho (X/Z)=1$ and both $-(K_X+S+B)$ and $-S$ are ample over $Z$, it follows that $K_X+S+B\sim _{Z,\Q}\mu S$ for some $\mu>0$ and so 
\begin{equation} \label{eq:support_of_the_mmp_divisor} 
K_{X_1} + \Delta_1 \sim_{Z,\Q} \lambda(K_{X_1}+S_1+aE_1+B_1)\sim _{Z,\Q} \lambda \mu S_1+\lambda ' E_1, 
\end{equation}
where $\lambda' \geq 0$. Note that $\lambda'>0$ if $g(E_1) \subseteq S$ and $\lambda'=0$ if $g(E_1) \not \subseteq S$.
\begin{claim}\label{c-3} {There exists a sequence of $(K_{X_1}+\Delta_1)$-flips $X_1 \dasharrow \ldots \dasharrow X_n$ over $Z$ such that either $X_n$ admits a $(K_{X_n}+\Delta_n)$-negative contraction of $E_n$ of relative Picard rank one, or $K_{X_n}+\Delta_n$ is semiample with the associated fibration contracting $E_n$. Here $\Delta_n$ and $E_n$ are strict transforms of $\Delta_1$ and $E_1$, respectively.} 
\end{claim}
In the course of the proof we will show that the qdlt-ness of $(X_1,S_1+E_1+B_1^c)$ is preserved (see Lemma \ref{lemma:flopping_dlt}) except possibly at the very last step before the contraction takes place. Therefore, all the flips in this MMP exist by Proposition \ref{proposition:building_blocks_qdlt}.
\begin{proof} 
Let $f_i \colon X_i \to Z$ be the induced map to $Z$. Since we work over a sufficiently small neighbourhood of $Q \in Z$, we can assume that all the flipped curves are contracted to $Q$ under $f_i$, and so $X_1 \dashrightarrow X_n$ is an isomorphism over $Z \, \backslash \, \{Q\}$.   Let $(X_i, \Delta_i)$ and $(X_i, S_i + E_i + B^c_i)$ be the appropriate strict transforms. The latter pair is a $6$-complement of $(X_i,S_i+E_i+B_i)$, where the strict transforms $B_i$ of $B_1$ have standard coefficients. Note that $E_1$ is not contracted as $X_1 \dasharrow \ldots \dasharrow X_n$ is a sequence of flips, thus inducing an isomorphism on the generic point of $E_1$.  

Suppose that $K_{X_n}+\Delta_n$ is nef. There are two cases: either $g(E_1) \subseteq S$ and $f_1(E_1)=Q$, or $g(E_1) \not \subseteq S$. We claim that the former cannot happen. Indeed, assume that $f_1(E_1)=Q$ and let $\pi_1 \colon W \to X_1$ and  $\pi_n \colon W \to X_n$ be a common resolution of $X_1$ and $X_n$ such that $\pi_1$ and $\pi_n$ are isomorphisms over $Z \, \backslash \, \{Q\}$. Since $K_{X_n} + \Delta_n$ is nef and $K_{X_1} + \Delta_1$ is anti-nef (but not numerically trivial) over $Z$, 
\[
\pi_n^*(K_{X_n} + \Delta_n) - \pi_1^*(K_{X_1} + \Delta_1)
\]
is exceptional, nef, and anti-effective over $Z$ by the negativity lemma. Moreover, its support must be equal to the whole exceptional locus over $Z$  as it is non-empty and contracted to $Q$ under the map to $Z$ (cf.\ \cite[Lemma 3.39(2)]{km98}). This is impossible, because $E_1$ is not contained in its support while $f_1(E_1)=Q$. 

Now, assuming that $g(E_1) \not \subseteq S$ is a curve intersecting $S$, we will show that $K_{X_n} +\Delta_n \sim_{\Q,Z} \lambda \mu S_n$ is semiample. Let $G := f_n^{-1}(P)$ for a (non-necessarily closed) point $P \in Z$. By \cite[Theorem 1.1]{CT17} it is enough to show that $S_n|_G$ is semiample. Since $X_1 \dashrightarrow X_n$ is an isomorphism over $Z \, \backslash \, \{Q\}$, $S_1 = g^*S$, and $S$ is semiample over $Z \, \backslash \, \{Q\}$,  we get that $S_n|_G$ is semiample when $P \neq Q$. Thus, we may assume that $P=Q$. By \cite[Theorem]{keel99}, it is enough to verify that $S_n|_{\mathbb{E}(S_n|_G)}$ is semiample. Since $G$ is one-dimensional, every connected component of $\mathbb{E}(S_n|_G) \subseteq G$ is either entirely contained in $S_n$ or is disjoint from it. In particular, it is enough to show that $S_n|_{S_n}$, or equivalently $(K_{X_n}+\Delta_n)|_{S_n}$, is semiample. Recall that $S_n \subseteq \lfloor \Delta_n \rfloor$, and so $K_{\tilde S_n} + \Delta_{\tilde S_n} = (K_{X_n}+\Delta_n)|_{\tilde S_n}$ is semiample by \cite[Theorem 1.1]{tanakaimperfect}, where $\tilde S_n$ is the normalisation of $S_n$. Since $\tilde S_n \to S_n$ is a universal homeomorphism (see \cite[Theorem 1.2]{HW18}), $(K_{X_n}+\Delta_n)|_{S_n}$ is semiample and so is $K_{X_n} +\Delta_n$. Since  $(K_{X_n}+\Delta_n)|_{E_n}$ is relatively numerically trivial over $Z \, \backslash \, \{Q\}$ (as so is $(K_{X_1}+\Delta_1)|_{E_1}$), we get that the associated semiample fibration contracts $E_n$.\\

From now on, $K_{X_n}+\Delta_n$ is not nef. In order to run the MMP, we assume that $(X_n,S_n + E_n + B_n^c)$ is qdlt by induction. The cone theorem is valid by \cite{keel99} (cf.\ \cite[Theorem 2.4]{HW18}). Pick $\Sigma_n$ a $(K_{X_n}+\Delta_n)$-negative extremal curve. By (\ref{eq:support_of_the_mmp_divisor}), we have $K_{X_n} + \Delta_n \sim _{Z,\Q} \lambda \mu S_n+\lambda ' E_n$, and so $\Sigma_n \cdot S_n < 0$ or $\Sigma_n \cdot E_n < 0$. The contraction of $\Sigma_n$ exists by \cite[Theorem 1.2 and Proposition 2.6]{HW18} applied to $(X_n,\Delta_n)$ in the former case and to $(X_n, S_n+E_n + B_n)$ in the latter (\cite[Theorem 1.2 and Proposition 2.6]{HW18} assumes that the singularities are dlt, but we can immediately reduce the qdlt case to the plt case by making the coefficients smaller).

If the corresponding contraction is divisorial, then we are done  as it must contract $E_n$. Hence, we can assume that $\Sigma_n$ is a flipping curve. If $E_n \cdot \Sigma_n \leq 0$, then $-(K_{X_n}+ S_n + B_n + E_n)$ has standard coefficients, is qdlt, and ample over the contraction of $\Sigma_n$, so the flip exists by Proposition \ref{proposition:building_blocks_qdlt} as $(X_n,S_n+E_n+B^c_n)$ is a $6$-complement. If $E_n \cdot \Sigma_n > 0$, then the flip exists by Proposition \ref{proposition:flip_case3} applied to $(X_n, \Delta_n)$. \\

To conclude the proof we shall  show that $(X_{n+1}, S_{n+1} + E_{n+1} + B^c_{n+1})$ is qdlt unless $X_{n+1}$ admits a contraction of $E_{n+1}$.  By Lemma \ref{lemma:flopping_dlt}, we can suppose that $S_{n+1} \cap E_{n+1} = \emptyset$ and aim for showing that the sought-for contraction exists. 

Let $\Sigma'$ be a curve which is exceptional over $Q\in Z$, contained neither in $S_{n+1}$ nor $E_{n+1}$, but intersecting $S_{n+1}$ (it exists by connectedness of the exceptional locus over $Q\in Z$, and the fact that both $S_{n+1}$ and $E_{n+1}$ intersect this exceptional locus), and let $C \subseteq E_{n+1}$ be any exceptional curve such that $C \cdot E_{n+1} <0$ (it exists by the negativity lemma as $E_{n+1}$ is exceptional over $Z$). We claim that $C' \cdot S_{n+1}  > 0$ for every exceptional curve $C' \not \subseteq E_{n+1}$. To this end, assume by contradiction that there  exists $C' \not \subseteq E_{n+1}$ satisfying $C' \cdot S_{n+1} \leq 0$. Since $\rho(X_{n+1}/Z)=2$, we get that
\[
C' \equiv  aC + b\Sigma',
\]
for $a, b \in \R$. Given $C \cdot S_{n+1} = 0$ and $\Sigma' \cdot S_{n+1} > 0$, we have $b \leq 0$. As $C' \cdot E_{n+1}\geq 0$, $C \cdot E_{n+1} < 0$, and $\Sigma' \cdot E_{n+1} \geq 0$, we have $a \leq 0$. Therefore, for an ample divisor $A$ we have
\[
0 < C' \cdot A = (aC + b\Sigma') \cdot A \leq 0 
\] 
which is a contradiction. 

Since $S_{n+1} \cap {E_{n+1}}$ is empty, $S_{n+1}$ is thus nef and $\mathbb{E}(S_{n+1}) \subseteq E_{n+1}$ (see \cite{CT17} for the definition of $\mathbb{E}$ in the relative setting). Hence $S_{n+1}$ is semiample by \cite[Proposition 2.20]{CT17} and induces a  contraction of $E_{n+1}$. It does not contract $\Sigma'$, and so is of relative Picard rank one. Moreover, 
\[
(K_{X_{n+1}}+\Delta_{n+1}) \cdot C \sim_{Z,\Q} \mu\lambda S_{n+1} \cdot C + \lambda' E_{n+1} \cdot C = \lambda' E_{n+1} \cdot C \leq 0,
\]
and so either $\lambda'=0$ and $K_{X_{n+1}}+\Delta_{n+1}\sim_{Z,\Q} \mu\lambda S_{n+1}$ is semiample with the associated fibration contracting $E_{n+1}$, or $\lambda' > 0$, $(K_{X_{n+1}}+\Delta_{n+1}) \cdot C < 0$, and so the above contraction is a $(K_{X_{n+1}}+\Delta_{n+1})$-negative Mori contraction of relative Picard rank one.

\end{proof}

Let $\phi \colon X_{n} \to X^+$ be the contraction of $E_n$ as in the claim, let $\Delta^+ := \phi_* \Delta_{n}$, let $S^+ := \phi_* S_n$, and let $B^+ := \phi_*B_n$. The projection onto $Z$ factors through a small contraction $\pi^+ \colon X^+ \to Z$ and  $\rho(X^{+}/Z)\leq 1$. Recall that 
\[
K_{X_n}+ \Delta_n \sim_{Z,\Q} \lambda(K_{X_n}+S_n+aE_n+B_n)\sim _{Z,\Q} \lambda \mu S_n+\lambda ' E_n. 
\]
Since $\phi$ is either $(K_{X_n}+S_n+aE_n+B_n)$-negative of Picard rank one, or $(K_{X_n}+S_n+aE_n+B_n)$-trivial, the discrepancies of $(X^+, S^+ + B^+)$ are not smaller than those of $(X_n,S_n+aE_n+B_n)$. Moreover, since $K_{X_1}+S_1+aE_1+B_1$ is anti-nef over $Z$ and not numerically trivial, at least one step of the $(K_{X_1}+\Delta_1)$-MMP (equivalently, $(K_{X_1}+S_1+aE_1+B_1)$-MMP) has been performed (that is, $n \geq 2$, or $\phi$ is a $(K_{X_n} + \Delta_n)$-negative contraction of $E_n$). In particular, there exists a divisorial valuation for which the discrepancy of $(X^+,S^+ + B^+)$ is higher than the discrepancy of $(X_1,S_1+aE_1+B_1)$, which in turn coincide with the discrepancy of $(X,S+B)$.

Therefore, $K_{X^+} + \Delta^+$ cannot be relatively anti-ample, because then $(X^+, S^+ + B^+)$ would be isomorphic to $(X, S+B)$,  which is impossible as the MMP has increased the discrepancies. If $K_{X^+}+\Delta^+$ is relatively numerically trivial, then we claim that $K_{X^+}+\Delta^+ \sim_{Z, \Q} 0$. Indeed,
\[
K_{X^+} + \Delta^+ \sim_{Z,\Q} \lambda \mu S^+,
\]
for $\lambda, \mu > 0$, and since $S^+$ intersects the exceptional locus, we must in fact have that $\Supp \mathrm{Exc}(\pi^+) \subseteq S^+$. By \cite[Proposition 2.20]{CT17}, it is thus enough to show that $K_{\tilde S^+} + \Delta_{\tilde S^+} = (K_{X^+} + \Delta^+)|_{\tilde S^+}$ is semiample, where $\tilde S^+ \to S^+$ is the normalisation of $S^+$, which in turn follows from \cite[Theorem 1.1]{tanakaimperfect}. Here, we used the fact that $\tilde S^+ \to S^+$ is a universal homeomorphism (see \cite[Theorem 1.2]{HW18}). As a consequence, $S^+$ descends to $Z$. This is impossible as its image (equal to the image of $S$) in $Z$ is not $\Q$-Cartier. 

Therefore, $K_{X^+} + \Delta^+$ is relatively ample, and so
$X^+ \to Z$ is the flip of $X\to Z$ by \cite[Corollary 6.4]{km98}.
\end{proof}

\subsection{The proof of Theorem \ref{theorem:flips}}

Given Theorem \ref{theorem:almost_flips}, the following proof follows the same strategy as in \cite[Theorem 6.3]{birkar13}. For the convenience of the reader, we provide a brief sketch of Birkar's argument in the projective case.
\begin{proof}[Proof of Theorem \ref{theorem:flips}]
First, we can assume that every component $S$ of $\Supp \Delta$ is relatively antiample. Further, let $\zeta(\Delta)$ be the number of components of $\Delta$ with coefficients not in the set  $\Gamma := \{1\} \cup \{1 - \frac{1}{n} \mid n >0 \}$. If $\zeta(\Delta)=0$, then the flip exists by Theorem \ref{theorem:almost_flips}. By induction, we can assume that the flip exists for all flipping contractions of log pairs $(X', \Delta')$ such that $\zeta(\Delta') < \zeta(\Delta)$.

By replacing $\Delta$ with $\Delta - \frac{1}{l}\lfloor \Delta \rfloor$ for $l\gg 0$, we can assume that $(X,\Delta)$ is klt without changing $\zeta(\Delta)$. Write $\Delta = aS + B$, where $S \not \subseteq \Supp B$ and $a \not \in \Gamma$. Let $\pi \colon W \to X$ be a log resolution of $(X,S+B)$ with exceptional divisor $E$ and set $B_W := \pi^{-1}_*B + E$. Since $K_X + \Delta \equiv_Z \mu S$ for some $\mu > 0$, we have that
\begin{align*}
K_W+S_W+B_W &= \pi^*(K_X+\Delta) + (1-a)S_W + F\\
 			&\equiv_Z (1-a+\mu)S_W + F',
\end{align*}
where $S_W := \pi^{-1}_*S$, and $F$, $F'$ are effective $\Q$-divisors  exceptional over $X$.

Run a $(K_W+S_W+B_W)$-MMP over $Z$. By induction all flips exist in this MMP  as $\zeta(S_W+B_W)<\zeta(\Delta)$. Moreover, by the above equation, every extremal ray is negative on $(1-a+\mu)S_W + F'$ and hence on an irreducible component of $\lfloor S_W + B_W \rfloor$. In particular, all contractions exist by \cite[Theorem 1.2 and Proposition 2.6]{HW18}. The cone theorem is valid by a result of Keel (see e.g.\ \cite[Theorem 2.4]{HW18}). Let $h \colon W \dasharrow Y$ be an output of this MMP and let $S_Y$, $B_Y$, and $F_Y$ be the strict transforms of $S_W$, $B_W$, and $F$, respectively. 

Now, run a $(K_Y + aS_Y + B_Y)$-MMP over $Z$ with scaling of $(1-a)S_Y$. In particular, if $R$ is an extremal ray, then $R \cdot S_Y > 0$ and 
\[
(K_Y+B_Y)\cdot R < 0.
\]
As $\zeta(B_Y) < \zeta(\Delta)$, all the flips in this MMP exist by induction. By the same argument as in the above paragraph, the cone theorem is valid in this setting and all contractions exist. Let $(X^+, aS^+ + B^+)$ be an output of this MMP. We claim that this is the flip of $(X,aS+B)$.

To this end, we notice that the negativity lemma applied to a common resolution $\pi_1 \colon W' \to X$ and $\pi_2 \colon W' \to X^+$ implies that
\[
\pi_1^*(K_X+aS+B) - \pi_2^*(K_{X^+} + aS^+ + B^+) \geq 0.
\]       
Since $(X,aS+B)$ is klt, this shows that $\lfloor B^+ \rfloor = 0$ and all the divisors in $E$ were contracted. In particular, $X \dashrightarrow X^+$ is an isomorphism in codimension one. We claim that $K_{X^+}+aS^++B^+$ is relatively ample over $Z$ and so $(X^+,aS^++B^+)$ is the flip of $X$. 

To this end, we note that $\rho(X^+/Z)=1$ (cf.\ \cite[Lemma 1.6]{VHK7}). Indeed, 
\[
\rho(W/X^+) + \rho(X^+/Z) = \rho(W/Z) = \rho(W/X) + \rho(X/Z)
\]
and $\rho(W/X)=\rho(W/X^+)$ is equal to the number of exceptional divisors. Since $\rho(X/Z)$ is equal to one, so is $\rho(X^+/Z)$. In particular, to conclude the proof of the theorem it is enough to show that $K_{X^+}+aS^++B^+$ cannot be relatively numerically trivial over $Z$. Assume by contradiction, that it is relatively numerically trivial. Then
\[
\pi_1^*(K_X+aS+B) - \pi_2^*(K_{X^+} + aS^+ + B^+)
\]   
is exceptional and relatively numerically trivial over $X$. Thus, it is empty by the negativity lemma which contradicts the fact that it is exceptional and non-numerically trivial over $Z$. 
\end{proof}

Theorem \ref{thm:minimal_models}, Theorem \ref{thm:bpf}, and Theorem \ref{thm:cone} now follow by exactly the same proof as \cite[Theorem 1.5 and 1.7]{BW14}, \cite[Theorem 1.2]{BW14}, and \cite[Theorem 1.1]{BW14}, respectively.

\section*{Acknowledgements}{}
We would like to thank Paolo Cascini, James M$^{c}$Kernan, Karl\linebreak Schwede, and Hiromu Tanaka for comments and helpful suggestions.

The first author was partially supported by NSF research grants no: DMS-1300750, DMS-1840190, DMS-1801851 and by a grant
from the Simons Foundation; Award Number: 256202. He would also like
to thank the Mathematics Department and the Research Institute for Mathematical Sciences,
located Kyoto University and the Mathematical Sciences Research Institute in Berkeley were some of this research was conducted. The second author was supported by the Engineering and Physical Sciences Research Council [EP/L015234/1] during his PhD at Imperial College London, by the National Science Foundation under Grant No.\ DMS-1638352 at the Institute for Advanced Study in Princeton, and by the National Science Foundation under Grant No.\ DMS-1440140 while the author was in residence at the Mathematical Sciences Research Institute in Berkeley, California, during the Spring 2019 semester.

\bibliographystyle{amsalpha}
\bibliography{final}

\begin{bibdiv}
\begin{biblist}

\bib{VHK7}{article}{
      author={Alexeev, Valery},
      author={Hacon, Christopher},
      author={Kawamata, Yujiro},
       title={Termination of (many) 4-dimensional log flips},
        date={2007},
        ISSN={0020-9910},
     journal={Invent. Math.},
      volume={168},
      number={2},
       pages={433\ndash 448},
         url={https://doi.org/10.1007/s00222-007-0038-1},
      review={\MR{2289869}},
}

\bib{bchm06}{article}{
      author={Birkar, C.},
      author={Cascini, P.},
      author={Hacon, C.},
      author={M\textsuperscript{c}Kernan, J.},
       title={Existence of minimal models for varieties of log general type},
        date={2010},
     journal={J. Amer. Math. Soc.},
      volume={23},
      number={2},
       pages={405\ndash 468},
}

\bib{birkar13}{article}{
      author={Birkar, Caucher},
       title={Existence of flips and minimal models for 3-folds in char {$p$}},
        date={2016},
        ISSN={0012-9593},
     journal={Ann. Sci. \'Ec. Norm. Sup\'er. (4)},
      volume={49},
      number={1},
       pages={169\ndash 212},
         url={http://dx.doi.org/10.24033/asens.2279},
}

\bib{BW14}{article}{
      author={Birkar, C.},
      author={Waldron, J.},
       title={Existence of {M}ori fibre spaces for 3-folds in {${\rm
  char}\,p$}},
        date={2017},
        ISSN={0001-8708},
     journal={Adv. Math.},
      volume={313},
       pages={62\ndash 101},
         url={https://doi.org/10.1016/j.aim.2017.03.032},
}

\bib{CT06PLT}{article}{
      author={Cascini, Paolo},
      author={Tanaka, Hiromu},
       title={{Purely log terminal threefolds with non-normal centres in
  characteristic two}},
        date={2016},
     journal={arXiv:1607.08590v1},
}

\bib{CT17}{article}{
      author={Cascini, Paolo},
      author={Tanaka, Hiromu},
       title={Relative semi-ampleness in positive characteristic},
        date={2017},
     journal={arXiv:1706.04845},
}

\bib{CTW}{article}{
      author={Cascini, Paolo},
      author={Tanaka, Hiromu},
      author={Witaszek, Jakub},
       title={On log del {P}ezzo surfaces in large characteristic},
        date={2017},
        ISSN={0010-437X},
     journal={Compos. Math.},
      volume={153},
      number={4},
       pages={820\ndash 850},
         url={http://dx.doi.org/10.1112/S0010437X16008265},
}

\bib{CTW15a}{article}{
      author={Cascini, Paolo},
      author={Tanaka, Hiromu},
      author={Witaszek, Jakub},
       title={Klt del {P}ezzo surfaces which are not {G}lobally {F}-split},
        date={2018},
     journal={Int. Math. Res. Not. IMRN},
      number={7},
       pages={2135\ndash 2155},
}

\bib{ctx13}{article}{
      author={Cascini, P.},
      author={Tanaka, H.},
      author={Xu, C.},
       title={On base point freeness in positive characteristic},
        date={2015},
     journal={Ann. Sci. Ecole Norm. Sup.},
      volume={48},
      number={5},
       pages={1239\ndash 1272},
}

\bib{dFKX}{incollection}{
      author={de~Fernex, Tommaso},
      author={Koll\'{a}r, J\'{a}nos},
      author={Xu, Chenyang},
       title={The dual complex of singularities},
        date={2017},
   booktitle={Higher dimensional algebraic geometry---in honour of {P}rofessor
  {Y}ujiro {K}awamata's sixtieth birthday},
      series={Adv. Stud. Pure Math.},
      volume={74},
   publisher={Math. Soc. Japan, Tokyo},
       pages={103\ndash 129},
      review={\MR{3791210}},
}

\bib{DW19}{article}{
      author={Das, Omprokash},
      author={Waldron, Joe},
       title={On the log minimal model program for 3-folds over imperfect
  fields of characteristic $p>5$},
        date={2019},
     journal={arXiv:1911.04394},
}

\bib{GNT06}{article}{
      author={Gongyo, Yoshinori},
      author={Nakamura, Yusuke},
      author={Tanaka, Hiromu},
       title={{Rational points on log Fano threefolds over a finite field}},
        date={2016},
     journal={Journal of the European Mathematical Society (to appear)},
}

\bib{HNT17}{article}{
      author={Hashizume, Kenta},
      author={Nakamura, Yusuke},
      author={Tanaka, Hiromu},
       title={Minimal model program for log canonical threefolds in positive
  characteristic},
        date={2017},
     journal={arXiv:1711.10706 (to appear in Math.\ Res.\ Lett)},
}

\bib{HW17}{article}{
      author={Hacon, Christopher},
      author={Witaszek, Jakub},
       title={On the rationality of {K}awamata log terminal singularities in
  positive characteristic},
        date={2019},
     journal={Algebraic Geometry},
      volume={6},
      number={5},
       pages={516\ndash 529},
}

\bib{HW18}{article}{
      author={Hacon, Christopher},
      author={Witaszek, Jakub},
       title={On the relative {M}inimal {M}odel {P}rogram for threefolds in low
  characteristics},
        date={2019},
     journal={arXiv:1909.12872},
}

\bib{hx13}{article}{
      author={Hacon, C.},
      author={Xu, C.},
       title={On the three dimensional minimal model program in positive
  characteristic},
        date={2015},
        ISSN={0894-0347},
     journal={J. Amer. Math. Soc.},
      volume={28},
      number={3},
       pages={711\ndash 744},
         url={http://dx.doi.org/10.1090/S0894-0347-2014-00809-2},
}

\bib{keel99}{article}{
      author={Keel, S.},
       title={Basepoint freeness for nef and big line bundles in positive
  characteristic},
        date={1999},
     journal={Ann. of Math. (2)},
      volume={149},
      number={1},
       pages={253\ndash 286},
}

\bib{km98}{book}{
      author={Koll{\'a}r, J.},
      author={Mori, S.},
       title={Birational {G}eometry of {A}lgebraic {V}arieties},
      series={Cambridge {T}racts in {M}athematics},
   publisher={Cambridge University Press},
        date={1998},
      volume={134},
}

\bib{KollarSing}{book}{
      author={Koll{\'a}r, J{\'a}nos},
       title={Singularities of the minimal model program},
      series={Cambridge Tracts in Mathematics},
   publisher={Cambridge University Press, Cambridge},
        date={2013},
      volume={200},
        ISBN={978-1-107-03534-8},
         url={http://dx.doi.org/10.1017/CBO9781139547895},
}

\bib{Nakamura19}{article}{
      author={Nakamura, Yusuke},
       title={Dual complex of log {F}ano pairs and its application to {W}itt
  vector cohomology},
        date={2019},
     journal={arXiv:1903.07248},
}

\bib{schwede14}{article}{
      author={Schwede, K.},
       title={A canonical linear system associated to adjoint divisors in
  characteristic {$p>0$}},
        date={2014},
        ISSN={0075-4102},
     journal={J. Reine Angew. Math.},
      volume={696},
       pages={69\ndash 87},
         url={http://dx.doi.org/10.1515/crelle-2012-0087},
}

\bib{SS}{article}{
      author={Schwede, K.},
      author={Smith, K.~E.},
       title={Globally {$F$}-regular and log {F}ano varieties},
        date={2010},
        ISSN={0001-8708},
     journal={Adv. Math.},
      volume={224},
      number={3},
       pages={863\ndash 894},
         url={http://dx.doi.org/10.1016/j.aim.2009.12.020},
}

\bib{tanaka12}{article}{
      author={Tanaka, H.},
       title={Minimal models and abundance for positive characteristic log
  surfaces},
        date={2014},
        ISSN={0027-7630},
     journal={Nagoya Math. J.},
      volume={216},
       pages={1\ndash 70},
         url={http://dx.doi.org/10.1215/00277630-2801646},
}

\bib{tanakaimperfect}{article}{
      author={Tanaka, H.},
       title={Abundance theorem for surfaces over imperfect fields},
        date={2016},
     journal={arXiv:1502.01383},
}

\bib{tanaka16_excellent}{article}{
      author={Tanaka, H.},
       title={{Minimal Model Programme for excellent surfaces}},
        date={2016},
     journal={arXiv:1608.07676v2 (to appear in Annales de l'Institut Fourier)},
}

\bib{waldronlc}{article}{
      author={Waldron, Joe},
       title={The {LMMP} for log canonical 3-folds in characteristic $p>5$},
        date={2018},
        ISSN={0027-7630},
     journal={Nagoya Math. J.},
      volume={230},
       pages={48\ndash 71},
}

\bib{watanabe91}{article}{
      author={Watanabe, K.},
       title={{$F$}-regular and {$F$}-pure normal graded rings},
        date={1991},
        ISSN={0022-4049},
     journal={J. Pure Appl. Algebra},
      volume={71},
      number={2-3},
       pages={341\ndash 350},
         url={http://dx.doi.org/10.1016/0022-4049(91)90157-W},
}

\bib{witaszek15}{article}{
      author={Witaszek, Jakub},
       title={Effective bounds on singular surfaces in positive
  characteristic},
        date={2017},
        ISSN={0026-2285},
     journal={Michigan Math. J.},
      volume={66},
      number={2},
       pages={367\ndash 388},
         url={https://doi.org/10.1307/mmj/1491465684},
}

\end{biblist}
\end{bibdiv}


\begin{thebibliography}{asdf}
\bibitem[CT17]{CT17} Cascini, Paolo, Tanaka, Hiromu, {\it Relative semi-ampleness in positive characteristic}, {2017},
 {arXiv:1706.04845}.
\bibitem[dFKX]{dFKX}T. de Fernex, J. Koll\'ar, C. Xu {\it The dual complex of singularities}
\bibitem[HX]{HX}C. Hacon, C. Xu, {\it On the three dimensional minimal model
program in positive characteristic}
\bibitem[Kol]{KollarSing}J. Koll\'ar, {\it Singularities of the Minimal Model Program}
\bibitem[Kol92]{Kol92}J. Koll\'ar, {\it Flips and abundance for algebraic threefolds}
\bibitem[Corti07]{Corti07} A. Corti, {\it 3--fold flips after Shokurov.}
Flips for 3-folds and 4-folds, 18-
48, Oxford Lecture Ser. Math. Appl., 35, Oxford Univ. Press, Oxford, 2007.
\bibitem[CTW]{CTW} P. Cascini, H. Tanaka, J. Witaszek {\it On log del Pezzo surfaces in large characteristic}
\bibitem[SS]{SS} K. Smith, K. Schwede {\it Globally F-regular and log Fano varieties}
\bibitem[HW17]{HW17} C. Hacon, J. Witaszek, {\it On the rationality of kawamata log
terminal singularities in positive
characteristic.} https://arxiv.org/pdf/1706.03204.pdf
\bibitem[HW18]{HW18} C. Hacon, J. Witaszek, {\it On the relative Minimal Model Program for threefolds in low characteristics.}
\bibitem[Hart77]{Hart77} R. Hartshorne, {\it Algebraic Geometry.} Graduate Texts in Mathematics, {\bf 52}, Springer 1977.
\bibitem[Prokhorov]{Prokhorov} Y. Prokhorov, {\it complements.} https://arxiv.org/pdf/math/9912111.pdf
\end{thebibliography}

\end{document}